\documentclass{amsart}%
\usepackage{amssymb}
\usepackage{amsmath}
\usepackage{hyperref}
\usepackage[pagewise]{lineno}
\usepackage[american]{babel}
\usepackage{amsfonts}
\usepackage{graphicx}
\usepackage{appendix}
\usepackage{dsfont}
\newtheorem{theorem}{Theorem}[section]
\newtheorem{lemma}[theorem]{Lemma}
\newtheorem{corollary}[theorem]{Corollary}
\theoremstyle{definition}
\newtheorem{definition}[theorem]{Definition}
\newtheorem{example}[theorem]{Example}
\newtheorem{assumption}[theorem]{Assumption}
\newtheorem{proposition}[theorem]{Proposition}

\theoremstyle{remark}
\newtheorem{remark}[theorem]{Remark}
\numberwithin{equation}{section}

\newcommand{\1}{\mathds 1}

\begin{document}
\title[Quasi-ergodic limits for finite absorbing Markov chains]{Quasi-ergodic limits\\for finite absorbing Markov chains}
\author{Fritz Colonius}
\author{Martin Rasmussen}
\address[Fritz Colonius]{Institut f\"{u}r Mathematik, Universit\"{a}t Augsburg, 86159 Augsburg, Germany}
\email[Fritz Colonius]{fritz.colonius@math.uni-augsburg.de}
\address[Martin Rasmussen]{Department of Mathematics, Imperial College London, 180 Queen's Gate, London
SW7 2AZ, United Kingdom}
\email[Martin Rasmussen]{m.rasmussen@imperial.ac.uk}
\date{\today}
\thanks{Fritz Colonius gratefully acknowledges support by a Nelder Visiting Fellowship from the Department of Mathematics, Imperial College London. Martin Rasmussen was supported by the European Union's Horizon 2020 research and innovation programme under the Marie Skłodowska-Curie grant agreement No. 643073.}
\begin{abstract}
We present formulas for quasi-ergodic limits of finite absorbing Markov chains. Since the irreducible case has been solved in 1965 by Darroch and Seneta~\cite{DarrS65}, we focus on the reducible case, and our results are based on a very precise asymptotic analysis of the (exponential and polynomial) growth behaviour along admissible paths.
\end{abstract}
\maketitle

\section{Introduction\label{Section1}}

The long-term statistical behaviour of Markov chains is determined by their ergodic stationary measures, in the sense that the time average of an observable of the process converges to the space average of the observable with respect to the ergodic stationary measure. In the context of absorbing Markov chains, the function of a stationary measure is naturally replaced by a quasi-stationary measure, and a quasi-stationary measure describes a statistical equilibrium distribution conditioned on that the Markov chain is not absorbed. The field of quasi-stationary measures has been very active recently, see the monograph Collet, Martinez and San Martin \cite{ColMS13}, as well as Champagnat and Villemonais \cite{ChamV16, ChamV18}, and the survey M\'{e}l\'{e}ard and Villemonais \cite{Meleard12}, which, in particular, covers applications to ecology and population dynamics.

It is well known that, when analysing the long-term statistical behaviour of absorbed Markov chains, quasi-stationary measures do not have the same function as stationary measures for non-absorbed Markov chains, despite their natural correspondence. In many settings, the time average of an observable of an absorbed Markov chain exists (when conditioned to non-absorption of the Markov chain), but this quantity is in general not equal to the space average of the observable with respect to the quasi-stationary measure. It turns out that, when taking the space average, the quasi-stationary measure needs to be replaced by another measure, often called quasi-ergodic measure. This has been first established for irreducible finite Markov chains by Darroch and Seneta \cite{DarrS65}. For more general irreducible Markov processes, Breyer and Roberts \cite{BreyR99} analysed this systematically, and they showed that the quasi-ergodic measure is absolutely continuous with respect to the quasi-stationary measure; they also coined the term \emph{quasi-ergodic limits} \cite[Theorem 1]{BreyR99} for these (conditioned) ergodic limits (cf.~also Zhang, Li, and Song \cite{ZhanLS14} and He, Zhang, and Zhu \cite{HeZhaZhu19}). Such quasi-ergodic limits have recently been used to define and analyse so-called conditioned Lyapunov exponents that describe the dynamical behaviour of random dynamical systems in compact subsets of the phase space, see Engel, Lamb and Rasmussen \cite{Engel_19_1}.

The literature on quasi-ergodic measures and limits has exclusively focussed on irreducible stochastic processes so far, and in this article, we aim at contributing to an understanding of the reducible case. It turns out that the analysis of quasi-ergodic limits is much more complicated for reducible processes, and for this reason, we focus here on the simplest possible case, given by finite state absorbing Markov chains.

We consider a stochastic matrix $P\in \mathbb{R}^{(d+1)\times(d+1)}$ of the form
\begin{equation}
P=%
\begin{pmatrix}
1 & 0\\
R & Q
\end{pmatrix}
\,, \label{S_stochmatrixP}%
\end{equation}
where $0$ is a row vector of zeros, and $R\in\mathbb{R}^{d\times1}%
,Q\in\mathbb{R}^{d\times d}$ with $R,Q\not =0$ and $d\geq2$.

We denote by $(X_{i})_{i\in\mathbb{N}_{0}}$ the Markov chain associated to the
substochastic matrix $Q$ starting in a probability vector $\pi\in\mathbb{R}^{d}$. We suppose that all states $\{1,\dots,d\}$ are transient,
i.e.~the probability of return to some state when starting in that state is less that $1$, which is equivalent to saying that the eigenvalues of the matrix $Q$ lie inside the unit circle (and in particular $1$ is not an eigenvalue of $Q$). Thus, this stochastic process is absorbed almost surely with absorption time $T$, meaning that the absorption state~$0$ is reached at time $T$.

We are interested in the \emph{quasi-ergodic limit}
\begin{equation}
\lim_{n\rightarrow\infty}\mathbb{E}_{\pi}\Big[\frac{1}{n+1}\sum_{i=0}%
^{n}f(X_{i})\Big\vert T>n\Big], \label{S_goal}%
\end{equation}
where $f:\{1,\dots,d\}\rightarrow\mathbb{R}$ is a given observable. As we will show in Corollary \ref{P_Corollary1} below, the
expectation in \eqref{S_goal} is determined by the average time the Markov chain visits its states, and hence, we
have to determine the \emph{quasi-ergodic measure}
\begin{equation}
\lim_{n\rightarrow\infty}\mathbb{E}_{\pi}\big[\tfrac{1}{n+1}\#\big\{m\in
\{0,\dots,n\}:X_{m}=j\big\}\big|T>n\big] \quad \mbox{for } j\in\{1,\dots,d\}\,. \label{1.3}%
\end{equation}
The main results, Theorem \ref{Theorem3} and Theorem \ref{Theorem4},
provide formulas for this limit.

This paper is organised as follows. Section~\ref{Section2} provides useful representations of the expectation in \eqref{1.3}. These are based on
results from Darroch and Seneta \cite{DarrS65}; the proofs are postponed to the Appendix, and we note that the quasi-ergodic limits for the
the irreducible case follow easily from these representations. We consider the theoretical analysis of the reducible case in Section~\ref{sec3}. In Subsection~\ref{subsecadmpath}, we assume without loss
of generality that the reducible matrix $Q$ is given in Frobenius normal form.
This can be achieved by permutations of the rows and columns, and the Frobenius
normal form is unique up to certain permutations, see Gantmacher~\cite[Chapter XIII,~\S 4]{Gant59}). We use admissible paths to reformulate the formulas for the quasi-ergodic limit, and in Subsection~\ref{Section4},
the main results are stated and proved. Here, we crucially have to assume that in the Frobenius normal form, the submatrices in the diagonal are scalar if their Perron--Frobenius eigenvalue is smaller than the maximal Perron--Frobenius eigenvalue. Finally, we illustrate the theoretical results by means of several examples in Section~\ref{sec4}.

\textbf{Notation.} Probability vectors $\pi$ are row vectors, while all other
vectors in $\mathbb{R}^{n}$ are column vectors. In all spaces $\mathbb{R}^{n}$
we abbreviate $\1=(1,\dots,1)^{\top}$. The set of natural numbers is
denoted by $\mathbb{N}$ and $\mathbb{N}_{0}=\mathbb{N}\cup\{0\}$. The number
of elements of a finite set $A$ is denoted by
$\#A$. For $k\in\mathbb{N}$ and $m\in\mathbb{Z}$
\[
\Gamma_{k}(m):=\big\{(\eta_{1},\dots,\eta_{k})\in\mathbb{N}_{0}^{k}:\eta
_{1}+\cdots+\eta_{k}=m\big\}
\]
and note that $\Gamma_{k}(m)=\emptyset$ for $m<0$. Products with an empty
index set are defined as $\prod_{i\in\emptyset}x_{i}=1$.

\section{Quasi-ergodic limits in the irreducible case\label{Section2}}

In this section, we consider the substochastic matrix $Q$
from~\eqref{S_stochmatrixP}, and we present results from Darroch and Seneta
\cite{DarrS65} for quasi-ergodic limits of the form \eqref{S_goal} in the
special case when $Q$ is irreducible.

Recall that if the matrix $Q$ is irreducible, then it is either eventually
positive or cyclic. It follows from the Perron--Frobenius theorem that $Q$ has
a simple eigenvalue $\rho\in(0,1)$, called the Perron--Frobenius eigenvalue,
such that the absolute values of all other eigenvalues of $Q$ are equal to or
less than $\rho$. The left eigenvector to this eigenvalue, $u\in\mathbb{R}%
^{d}$ with $u^{\top}Q=\rho u^{\top}$, has only positive entries and describes a quasi-stationary measure
when normalised via $\sum_{i=1}^{d}u_{i}=1$, which we assume in the following. If $Q$ is eventually
positive, then the absolute values of all other eigenvalues are smaller that
$\rho$.

The following proposition is our starting point for deriving formulas for
quasi-ergodic limits. The proof is given in the Appendix. Denote
\begin{equation}
\pi_{j}(z):=\pi D_{j}(z)\quad\text{and}\quad Q_{j}(z):=QD_{j}(z)
\quad\mbox{for all } z\in\mathbb{R}\,, \label{P_Q_j}%
\end{equation}
where $D_{j}(z)$ is the $d\times d$ diagonal matrix whose $j$-th diagonal
element is $z$ and all other diagonal elements are equal to $1$.

\begin{proposition}
\label{P_prop1} Consider a substochastic matrix $Q\in\mathbb{R}^{d\times d}$,
and let $(X_{i})_{i\in\mathbb{N}_{0}}$ be the associated Markov chain starting
in $\pi$. Then the following statements hold.

\begin{itemize}
\item[(i)] For all $j\in\{1,\dots,d\}$ and $n\in\mathbb{N}$, we have%
\[
\mathbb{E}_{\pi}\big[\tfrac{1}{n+1}\#\big\{m\in\{0,\dots,n\}:X_{m}%
=j\big\}\big|T>n\big]=\frac{\frac{\mathrm{d}}{\mathrm{d}z}\pi_{j}(z)Q_{j}%
^{n}(z)\1 \Big|_{z=1}}{(n+1)\pi Q^{n}\1 }.
\]

\item[(ii)] Suppose that $Q$ is eventually positive. Then for $j\in
\{1,\dots,d\}$, we have
\[
\mathbb{E}_{\pi}\big[\tfrac{1}{n+1}\#\big\{m\in\{0,\dots,n\}:X_{m}%
=j\big\}\big|T>n\big]=u_{j}v_{j}+O(\tfrac{1}{n}) \quad\mbox{as } n\to
\infty\,,
\]
where $v$ is the positive right eigenvector of $Q$ for the Perron--Frobenius
eigenvalue $\rho$ normalised by $u^{\top}v=1$.

\item[(iii)] Suppose that $Q$ is cyclic with period $h\in\mathbb{N}%
\setminus\{1\}$. Then $Q^{h}$ is eventually positive, and for all
$j\in\{1,\dots,d\}$, we have for the left and right normalised eigenvectors
$u^{\top}$ and $v$ of $Q^{h}$ for the eigenvalue $\rho^{h}$ that
\[
\mathbb{E}_{\pi}\big[\tfrac{1}{hn+1}\#\big\{m\in\{0,\dots,hn\}:X_{m}%
=j\big\}\big|T>hn\big\}=u_{j}v_{j}+O(\tfrac{1}{hn}) \quad\mbox{as } n\to
\infty\,.
\]

\end{itemize}
\end{proposition}

The following corollary uses the above result for the average evaluation of an
observable. The formula provided in (i) below will be the basis of our further
analysis of the reducible case. It shows that, in particular, the probability
of the average evaluation of an observable $f$ is determined by the average
number of times that $X_{i}$ is in some state $j$. For the irreducible case,
assertion (ii) below concerns a formula for the quasi-ergodic limit involving
the normalised right and left eigenvectors for the Perron--Frobenius
eigenvalue $\rho$ of $Q$.

\begin{corollary}
\label{P_Corollary1} Consider a substochastic matrix $Q\in\mathbb{R}^{d\times
d}$, and let $(X_{i})_{i\in\mathbb{N}_{0}}$ be the associated Markov chain
starting in $\pi$. Then the following statements hold.

\begin{itemize}
\item[(i)] For all $n\in\mathbb{N}$, we have
\begin{align*}
\mathbb{E}_{\pi}\left[  \tfrac{1}{n+1}\sum_{i=0}^{n}f(X_{i}%
)\Big\vert T>n\right]   &  =\sum_{j=1}^{d}f(j)\mathbb{E}_{\pi}\big[\tfrac
{1}{n+1}\#\big\{m\in\{0,\dots,n\}:X_{m}=j\big\}\big|T>n\big]\\
&  =\sum_{j=1}^{d}f(j)\frac{\frac{\mathrm{d}}{\mathrm{d}z}\pi_{j}(z)Q_{j}%
^{n}(z)\1 \Big|_{z=1}}{(n+1)\pi Q^{n}\1 }.
\end{align*}

\item[(ii)] If $Q$ is irreducible, then the quasi-ergodic limit is given by
\[
\lim_{n\rightarrow\infty}\mathbb{E}_{\pi}\Big[\frac{1}{n+1}\sum_{i=0}%
^{n}f(X_{i})\Big\vert T>n\Big]=\sum_{i=1}^{d}f(i)u_{i}v_{i},
\]
where $v$ and $u^{\top}$ are the right and left eigenvectors for the
eigenvalue $\rho$ of $Q$ normalised as in Proposition~\ref{P_prop1}~(ii).
\end{itemize}
\end{corollary}

\begin{proof}
(i) Using Proposition~\ref{P_prop1}~(i), one computes for fixed $n\in
\mathbb{N}$%
\begin{align*}
&  \sum_{j=1}^{d}f(j)\frac{\frac{\mathrm{d}}{\mathrm{d}z}\pi_{j}(z)Q_{j}%
^{n}(z)\1 \Big|_{z=1}}{(n+1)\pi Q^{n}\1 }\\
&  =\sum_{j=1}^{d}f(j)\mathbb{E}_{\pi}\big[\tfrac{1}{n+1}\#\big\{m\in
\{0,\dots,n\}:X_{m}=j\big\}\big|T>n\big]\\
&  =\tfrac{1}{n+1}\sum_{j=1}^{d}\mathbb{E}_{\pi}\big[f(j)\#\big\{m\in
\{0,\dots,n\}:X_{m}=j\big\}\big|T>n\big]\\
&  =\mathbb{E}_{\pi}\left[  \frac{1}{n+1}\sum_{i=0}^{n}f(X_{i}%
)\Big\vert T>n\right]  .
\end{align*}

(ii) Proposition~\ref{P_prop1}~(ii) and (iii) yield the assertion in the
irreducible case (where the matrix $Q$ is either eventually positive or cyclic).
\end{proof}

Corollary~\ref{P_Corollary1}~(ii) for irreducible $Q$ and $f(j)=j$ for $j\in
\{1,\dots,d\}$ is classical and has been proved in Darroch and
Seneta~\cite[p.~95]{DarrS65}. Here, the left eigenvector $u^{\top}$ is the
unique quasi-stationary measure, see van Doorn and Pollett \cite[Theorem~2.1]%
{vanDooPol09}. Thus, the quasi-ergodic limit is absolutely continuous with
respect to the quasi-stationary measure.

\section{Quasi-ergodic limits in the reducible case}\label{sec3}

While in the previous section, we obtained an quasi-ergodic limit formula for
irreducible matrices $Q$, we concentrate now on the reducible case, and we
suppose without loss of generality that the matrices $Q$ from
\eqref{S_stochmatrixP} are given in Frobenius normal form%
\begin{equation}
Q=%
\begin{pmatrix}
Q_{11} & 0 &  & 0\\
Q_{21} & Q_{22} &  & 0\\
&  & \ddots & \\
Q_{k1} & Q_{k2} &  & Q_{kk}%
\end{pmatrix}
\label{P_Q}%
\end{equation}
with matrices $Q_{ij}\in\mathbb{R}^{d_{i}\times d_{j}}$, where $d_{1}%
,\dots,d_{k}\in\mathbb{N}$. We assume in addition that the diagonal matrices
$Q_{ii}$ are eventually positive. The results for the general case, where the
diagonal matrices are irreducible (hence maybe periodic), are easy
consequences, see Remark~\ref{Remark_periodic} below.

We note that $\sum_{i=1}^{k}d_{i}=d$, and introduce index sets
\[
I_{j}=\big\{1+\textstyle\sum_{i=1}^{j-1}d_{i},\dots,\textstyle\sum_{i=1}%
^{j}d_{i}\big\}\quad\text{for all }\,j\in\{1,\dots,k\}\,,
\]
corresponding to the diagonal blocks of the matrix $Q$.

\subsection{Preparations and admissible paths}

\label{subsecadmpath}

In this subsection, we reformulate the quasi-ergodic problem using admissible
paths of indices. We denote the initial distribution by $\pi=(\pi_{1}%
,\dots,\pi_{k})$ with $\pi_{i}\in\mathbb{R}^{d_{i}}$ and first obtain a
version of Proposition~\ref{P_prop1} for the above systems in Frobenius normal form.

\begin{proposition}
\label{P_lemma2} Consider a matrix $Q$ of the form \eqref{P_Q}, and let
$(X_{i})_{i\in\mathbb{N}_{0}}$ be the Markov chain associated to the
substochastic matrix $Q$ starting in $\pi$. Then the following statements hold.

\begin{itemize}
\item[(i)] For $j\in\{1,\dots,d\}$ and $n\in\mathbb{N}$, the probability of
the average number of times that $X_{i}$ is in some state $j$ is
\[
\mathbb{E}_{\pi}\big[\tfrac{1}{n+1}\#\big\{m\in\{0,\dots,n\}:X_{m}%
=j\big\}\big|T>n\big]=\frac{\pi\left(  \sum_{r=0}^{n}Q^{r}e_{j}e_{j}^{\top
}Q^{n-r}\right)  \1 }{(n+1)\pi Q^{n}\1 }\,.
\]

\item[(ii)] For $\ell\in\{1,\dots,k\}$ and $n\in\mathbb{N}$, the probability
of the average number of times that $X_{i}$ is in some state in $I_{\ell}$ is
\[
\mathbb{E}_{\pi}\big[\tfrac{1}{n+1}\#\big\{m\in\{0,\dots,n\}:X_{m}\in I_{\ell
}\big\}\big|T>n\big]=\frac{\pi\left(  \sum_{r=0}^{n}Q^{r}\sum_{j\in I_{\ell}%
}e_{j}e_{j}^{\top}Q^{n-r}\right)  \1 }{(n+1)\pi Q^{n}\1 }\,.
\]

\end{itemize}
\end{proposition}

\begin{proof}
(i) Using \eqref{P_Q_j}, we compute
\begin{align*}
&  \frac{\mathrm{d}}{\mathrm{d}z}\pi_{j}(z)Q_{j}^{n}(z)\1 %
=\frac{\mathrm{d}}{\mathrm{d}z}\left(  \pi D_{j}(z)\left(  QD_{j}(z)\right)
^{n}\1 \right) \\
=  &  \pi e_{j}e_{j}^{\top}\left(  QD_{j}(z)\right)  ^{n}\1 +\pi
D_{j}(z)\left(  \sum_{r=1}^{n}\left(  QD_{j}(z)\right)  ^{r-1}Qe_{j}%
e_{j}^{\top}\left(  QD_{j}(z)\right)  ^{n-r}\right)  \1 \,.
\end{align*}
This implies that
\[
\frac{\mathrm{d}}{\mathrm{d}z}\pi_{j}(z)Q_{j}^{n}(z)\1 \Big|_{z=1}=\pi
e_{j}e_{j}^{\top}Q^{n}\1 +\pi\left(  \sum_{r=1}^{n}Q^{r}e_{j}%
e_{j}^{\top}Q^{n-r}\right)  \1 \,.
\]
Now the assertion follows from Proposition~\ref{P_prop1}~(i).

(ii) This follows from (i) and
\begin{align*}
&  \mathbb{E}_{\pi}\big[ \tfrac{1}{n+1}\#\big\{m\in\{0,\dots,n\}:X_{m}\in
I_{\ell}\big\}\big|T>n\big]\\
&  =\sum_{j\in I_{\ell}}\mathbb{E}_{\pi}\big[\tfrac{1}{n+1}\#\big\{m\in
\{0,\dots,n\}:X_{m}=j\big\}\big|T>n\big]\,,
\end{align*}
which finishes the proof of this proposition.
\end{proof}

We now aim at understanding the terms in Proposition~\ref{P_lemma2} better and
first note that for all $n\in\mathbb{N}_{0}$, we get
\[
Q^{n}=%
\begin{pmatrix}
Q_{11}^{(n)} & 0 &  & 0\\
Q_{21}^{(n)} & Q_{22}^{(n)} &  & 0\\
&  & \ddots & \\
Q_{k1}^{(n)} & Q_{k2}^{(n)} &  & Q_{kk}^{(n)}%
\end{pmatrix}
\,,
\]
where for $n\geq1$
\begin{equation}
Q_{ij}^{(n)}:=\sum_{\substack{s_{1},\dots,s_{n-1}=1,\dots,k\\i=s_{0}\geq
s_{1}\geq s_{2}\geq\cdots\geq s_{n-1}\geq s_{n}=j}}Q_{s_{0}s_{1}}Q_{s_{1}%
s_{2}}\cdots Q_{s_{n-1}s_{n}}\,, \label{P_Q_n}%
\end{equation}
and for $n=0$
\[
Q_{ij}^{(0)}:=\left\{
\begin{array}
[c]{l@{\quad:\quad}l}%
\operatorname{Id} & i=j\,,\\
0 & i\not =j\,.
\end{array}
\right.
\]
This follows by induction, since for $i\geq\ell$, the entries $Q_{i\ell
}^{(n+1)}$ of $Q^{n+1}=Q^{n}Q$ are given by
\begin{align*}
Q_{i\ell}^{(n+1)}  &  =\sum_{j=\ell}^{k}\sum_{\substack{s_{1},\dots
,s_{n-1}=1,\dots,k\\i\geq s_{1}\geq s_{2}\geq\cdots\geq s_{n-1}\geq
j}}Q_{is_{1}}Q_{s_{1}s_{2}}\cdots Q_{s_{n-1}j}Q_{j\ell}\\
&  =\sum_{\substack{s_{1},\dots,s_{n}=1,\dots,k\\i\geq s_{1}\geq s_{2}%
\geq\cdots\geq s_{n}\geq\ell}}Q_{is_{1}}Q_{s_{1}s_{2}}\cdots Q_{s_{n-1}s_{n}%
}Q_{s_{n}\ell}\,.
\end{align*}
We first consider the numerator in the formula from Proposition~\ref{P_lemma2}%
~(ii), which can be re-written as
\begin{align}
&  \allowdisplaybreaks\pi\left(  \sum_{r=0}^{n}Q^{r}\sum_{j\in I_{\ell}}%
e_{j}e_{j}^{\top}Q^{n-r}\right)  \1 \nonumber\\
&  =\allowdisplaybreaks\pi\sum_{r=0}^{n}%
\begin{pmatrix}
Q_{11}^{(r)} & 0 &  & 0\\
Q_{21}^{(r)} & Q_{22}^{(r)} &  & 0\\
&  & \ddots & \\
Q_{k1}^{(r)} & Q_{k2}^{(r)} &  & Q_{kk}^{(r)}%
\end{pmatrix}
\left(
\begin{array}
[c]{cccc}%
0 &  &  & 0\\
& \ddots &  & \\
Q_{\ell1}^{(n-r)} &  & Q_{\ell\ell}^{(n-r)} & 0\\
0 &  &  & 0
\end{array}
\right)  \1 \nonumber\\
&  \allowdisplaybreaks =\pi\sum_{r=0}^{n}\left(
\begin{array}
[c]{ccccc}%
0 &  &  &  & 0\\
& \ddots &  &  & \\
Q_{\ell\ell}^{(r)}Q_{\ell1}^{(n-r)} &  & Q_{\ell\ell}^{(r)}Q_{\ell\ell
}^{(n-r)} & 0 & 0\\
&  &  & \ddots & \\
Q_{k\ell}^{(r)}Q_{\ell1}^{(n-r)} &  & Q_{k\ell}^{(r)}Q_{\ell\ell}^{(n-r)} &  &
0
\end{array}
\right)  \1 \nonumber\\
&  \allowdisplaybreaks =(\pi_{\ell},\dots,\pi_{k})\sum_{r=0}^{n}\left(
\begin{array}
[c]{ccc}%
Q_{\ell\ell}^{(r)}Q_{\ell1}^{(n-r)} &  & Q_{\ell\ell}^{(r)}Q_{\ell\ell
}^{(n-r)}\\
& \ddots & \\
Q_{k\ell}^{(r)}Q_{\ell1}^{(n-r)} &  & Q_{k\ell}^{(r)}Q_{\ell\ell}^{(n-r)}%
\end{array}
\right)  \1 \nonumber\\
&  =\sum_{r=0}^{n}\sum_{i=\ell}^{k}\pi_{i}Q_{i\ell}^{(r)}\sum_{j=1}^{\ell
}Q_{\ell j}^{(n-r)}\1 =\sum_{i=\ell}^{k}\sum_{j=1}^{\ell}\pi_{i}%
\sum_{r=0}^{n}Q_{i\ell}^{(r)}Q_{\ell j}^{(n-r)}\1 . \label{P_3.6}%
\end{align}
We now aim at re-writing this product of certain sub-matrices of the matrix
$Q$ in a different way involving so-called admissible paths of indices.

\begin{definition}
[Admissible paths]\rule{0mm}{1mm}\vspace{-0.4cm}\newline

\begin{itemize}
\item[(i)] An \emph{admissible path }$\theta$ \emph{of length }$\kappa
=\kappa(\theta)$ is given by a finite and strictly decreasing sequence
$\theta=(\theta_{1},\theta_{2},\dots,\theta_{\kappa})$ such that $\theta
_{u}\in\{1,\dots,k\}$ and $Q_{\theta_{u}\theta_{u+1}}\not =0$ for all
$u\in\{1,\dots,\kappa-1\}$.

\item[(ii)] The \emph{set of admissible paths} is denoted by $\mathcal{P}$,
and we denote the set of admissible paths that go from $i$ to $j$ by
\[
\mathcal{P}_{ij}:=\big\{(\theta_{1},\theta_{2},\dots,\theta_{\kappa}%
)\in\mathcal{P}:\theta_{1}=i\mbox{ and }\theta_{\kappa}=j\big\}\,,
\]
and define the set of admissible paths through $\ell\in\{1,\dots,k\}$ as
\[
\mathcal{P}^{(\ell)}:=\big\{(\theta_{1},\dots,\theta_{\kappa})\in
\mathcal{P}:\mbox{ there exists a }u\in\{1,\dots,\kappa\}\text{ with }%
\theta_{u}=\ell\big\}\,.
\]

\end{itemize}
\end{definition}

We note that every finite sequence of natural numbers $s_{i}$ occurring in
non-zero products in sums of the form $Q_{i\ell}^{(r)}$ and $Q_{\ell
j}^{(n-r)}$, as defined in \eqref{P_Q_n}, must follow an admissible path. More
precisely, concentrating on $Q_{i\ell}^{(n)}$, for some $(s_{0},\dots,s_{n})$
such that $0\not =Q_{s_{0}s_{1}}Q_{s_{1}s_{2}}\cdots Q_{s_{n-1}s_{n}}$, there
exist a $\theta=(i,\theta_{2},\dots,\theta_{\kappa-1},\ell)\in\mathcal{P}%
_{i\ell}$ and exponents $\eta_{1},\dots,\eta_{\kappa}\in\mathbb{N}_{0}$ such
that $\sum_{u=1}^{\kappa}\eta_{u}=n+1-\kappa$ and
\begin{equation}
Q_{s_{0}s_{1}}Q_{s_{1}s_{2}}\cdots Q_{s_{n-1}s_{n}}=Q_{\theta_{1}\theta_{1}%
}^{\eta_{1}}Q_{\theta_{1}\theta_{2}}Q_{\theta_{2}\theta_{2}}^{\eta_{2}%
}Q_{\theta_{2}\theta_{3}}\cdots Q_{\theta_{\kappa-1}\theta_{\kappa}}%
Q_{\theta_{\kappa}\theta_{\kappa}}^{\eta_{\kappa}}\,. \label{P_Qs}%
\end{equation}
Every matrix which is subdiagonal in $Q$ occurs at most once, and for this
reason, most entries in this large matrix product are diagonal blocks
$Q_{\theta_{u}\theta_{u}}$ that are ordered with respect to $u$ and thus
appear as powers of these matrices.


The number of elements in%
\[
\Gamma_{\kappa}(m)=\big\{(\eta_{1},\dots,\eta_{\kappa})\in\mathbb{N}%
_{0}^{\kappa}:\eta_{1}+\cdots+\eta_{\kappa}=m\big\} \quad\mbox{for all }\kappa
\in\{1,\dots,k\} \mbox{ and } m\in\mathbb{N}
\]
is given by
\begin{equation}
\#\Gamma_{\kappa}(m)=\binom{\kappa+m-1}{m} \label{anzgamma}%
\end{equation}
(this is modelled by drawing $\kappa-1$ out of $m+1$ balls from an urn with
replacement and without order). For $\theta=(\theta_{1},\theta_{2}%
,\dots,\theta_{\kappa})\in\mathcal{P}$ and $m\in\mathbb{N}$, let%
\begin{equation}
Q(\theta,m):=\sum_{\eta\in\Gamma_{\kappa}(m+1-\kappa)}Q_{\theta_{1}\theta_{1}%
}^{\eta_{1}}Q_{\theta_{1}\theta_{2}}Q_{\theta_{2}\theta_{2}}^{\eta_{2}%
}Q_{\theta_{2}\theta_{3}}\cdots Q_{\theta_{\kappa-1}\theta_{\kappa}}%
Q_{\theta_{\kappa}\theta_{\kappa}}^{\eta_{\kappa}}\, \label{P_Q_theta}%
\end{equation}
and%
\[
Q(\theta,0):=\left\{
\begin{array}
[c]{l@{\quad:\quad}l}%
\operatorname{Id} & \theta\in\mathcal{P}_{ij} \mbox{, where }i=j\,,\\
0 & \theta\in\mathcal{P}_{ij} \mbox{, where } i\not =j\,.
\end{array}
\right.
\]
We use the following restrictions of $\theta=(\theta_{1},\dots,\theta_{\kappa
})\in\mathcal{P}^{(\ell)}$,
\[
\underline{\theta^{\ell}}:=(\theta_{1},\dots,\theta_{u}=\ell)\in
\mathcal{P}_{\theta_{1},\ell}\quad\text{ and }\quad\overline{\theta^{\ell}%
}:=(\theta_{u}=\ell,\dots,\theta_{\kappa})\in\mathcal{P}_{\ell,\theta_{\kappa
}}\,,
\]
and we write $\underline{\kappa}:=\underline{\kappa}(\theta,\ell):=u$ and
$\overline{\kappa}:=\overline{\kappa}(\theta,\ell):=\kappa-u+1$ for the length
of $\underline{\theta^{\ell}}$ and $\overline{\theta^{\ell}}$, respectively.
Hence, $\underline{\kappa}+\overline{\kappa}=\kappa+1$.

We obtain the following corollary to Proposition~\ref{P_lemma2}.

\begin{corollary}
\label{lemma2b} Consider a matrix $Q$ of the form \eqref{P_Q}, let
$(X_{i})_{i\in\mathbb{N}_{0}}$ be the Markov chain associated to the
substochastic matrix $Q$ starting in $\pi$, and let $\ell\in\{1,\dots,k\}$.
Then the following two statements hold.

\begin{itemize}
\item[(i)] We have
\begin{align*}
&  \mathbb{E}_{\pi}\left[  \tfrac{1}{n+1}\#\big\{m\in\{0,\dots,n\}:X_{m}\in
I_{\ell}\big\}\big|T>n\right] \\
&  \qquad\qquad=\frac{\sum_{\theta\in\mathcal{P}^{(\ell)}}\pi_{\theta_{1}}%
\sum_{r=0}^{n}Q(\underline{\theta^{\ell}},r)Q(\overline{\theta^{\ell}%
},n-r)\1 }{(n+1)\sum_{\theta\in\mathcal{P}}\pi_{\theta_{1}}%
Q(\theta,n)\1 }\,.
\end{align*}

\item[(ii)] For $s\in I_{\ell}$ we have with $t(s):=s-\sum_{i=0}^{\ell-1}%
d_{i}$ and $e_{t(s)}\in\mathbb{R}^{d_{\ell}},$%
\begin{align*}
&  \mathbb{E}_{\pi}\big[\tfrac{1}{n+1}\#\big\{m\in\{0,\dots,n\}:X_{m}%
=s\big\}\big|T>n\big]\\
&  =\frac{\sum_{\theta\in\mathcal{P}^{(\ell)}}\pi_{\theta_{1}}\sum_{r=0}%
^{n}Q(\underline{\theta^{\ell}},r)e_{t(s)}e_{t(s)}^{\top}Q(\overline
{\theta^{\ell}},n-r)\1 }{(n+1)\sum_{\theta\in\mathcal{P}}\pi
_{\theta_{1}}Q(\theta,n)\1 }.
\end{align*}

\end{itemize}
\end{corollary}

\begin{proof}
(i) First consider the denominator in Proposition~\ref{P_lemma2}~(ii). With
\eqref{P_Q_n}, we get%
\begin{align*}
&  \allowdisplaybreaks(n+1)\pi Q^{n}\1 =(n+1)\sum_{i=1}^{k}\pi_{i}%
\sum\limits_{j=1}^{i}Q_{ij}^{(n)}\1 \,\\
&  =(n+1)\sum_{i=1}^{k}\sum\limits_{j=1}^{i}\pi_{i}\sum_{i=s_{0}\geq s_{1}%
\geq\cdots\geq s_{n-1}\geq s_{n}=j}Q_{is_{1}}Q_{s_{1}s_{2}}\cdots Q_{s_{n-1}%
j}\1 \\
&  =\sum_{i=1}^{k}\sum_{j=1}^{i}(n+1)\pi_{i}\sum_{\theta\in\mathcal{P}_{ij}%
}Q(\theta,n)\1 \\
&  =(n+1)\sum\limits_{\theta\in\mathcal{P}}\pi_{\theta_{1}}Q(\theta
,n)\1 \,.
\end{align*}
Turning to the numerator we can write
\begin{equation}
\sum_{r=0}^{n}Q_{i\ell}^{(r)}Q_{\ell j}^{(n-r)}=\sum_{r=0}^{n}\Bigg(\sum
_{\theta\in\mathcal{P}_{i\ell}}Q(\theta,r)\Bigg)\Bigg(\sum_{\theta
\in\mathcal{P}_{\ell j}}Q(\theta,n-r)\Bigg)\,. \label{P_3.7}%
\end{equation}
Every admissible path $\theta\in\mathcal{P}^{(\ell)}\cap\mathcal{P}_{ij}$
corresponds to two admissible paths $\underline{\theta^{\ell}}\in
\mathcal{P}_{i\ell}$ and $\overline{\theta^{\ell}}\in\mathcal{P}_{\ell j}$.
Hence the numerator re-written in the form \eqref{P_3.6} is given by
\begin{align*}
&  \sum\limits_{i=\ell}^{k}\sum\limits_{j=1}^{\ell}\pi_{i}\sum\limits_{r=0}%
^{n}\bigg(\sum\limits_{\theta\in\mathcal{P}_{i\ell}}Q(\theta
,r)\bigg)\bigg(\sum\limits_{\theta\in\mathcal{P}_{\ell j}}Q(\theta
,n-r)\bigg)\1 \\
&  =\sum\limits_{\theta\in\mathcal{P}^{(\ell)}}\pi_{\theta_{1}}\sum
\limits_{r=0}^{n}Q(\underline{\theta^{\ell}},r)Q(\overline{\theta^{\ell}%
},n-r)\1 \,.
\end{align*}
(ii) Using Proposition~\ref{P_lemma2}~(i) and an appropriate modification of
formula \eqref{P_3.6}, one proves this analogously.
\end{proof}

\subsection{Formulas for quasi-ergodic limits\label{Section4}}

In this subsection, we determine formulas for quasi-ergodic limits for
matrices $Q$ of the form \eqref{P_Q}.

Recall that we assume that the diagonal matrices $Q_{ii}$ are eventually
positive and that the maximal eigenvalue of $Q_{ii}$ (the Perron--Frobenius
eigenvalue) is denoted by $\rho_{i}$. For $\theta=(\theta_{1},\dots
,\theta_{\kappa})\in\{1,\dots,k\}^{\kappa}$, we define $\rho(\theta
):=\max\{\rho_{\theta_{1}},\dots,\rho_{\theta_{\kappa}}\}$,
\begin{align*}
H^{+}(\theta)  &  :=\big\{u\in\{1,\dots,\kappa\}: \rho_{\theta_{u}}%
=\rho(\theta)\big \}\quad\text{ and }\\
H^{-}(\theta)  &  :=\big \{u\in\{1,\dots,\kappa\}:\rho_{\theta_{u}}%
<\rho(\theta)\big \}\,,
\end{align*}
and we denote the number of elements in these sets by $h^{+}(\theta):=\#
H^{+}(\theta)$ and $h^{-}(\theta):=\# H^{-}(\theta)$. Note that $h^{+}%
(\theta)+h^{-}(\theta) = \kappa=\kappa(\theta)$. In addition, we define
$\rho_{\max} := \max\{\rho_{1},\dots,\rho_{k}\}$ and $h^{+}_{\max}%
:=\max\{h^{+}(\theta): \theta\in\mathcal{P}\mbox{ and } \rho(\theta)=
\rho_{\max}\}$.

\begin{remark}
In the terminology of Friedland and Schneider \cite[p.~190]{FriSchn80}, if
$\rho_{i}=\rho_{\max}$, then $Q_{ii}$ determines a singular vertex of the
graph associated with $Q$, and the singular distance from $i$ to $j$ is given
by $h^{+}_{\max}-1$.
\end{remark}

We aim at quasi-ergodic limits by taking the limit $n\rightarrow\infty$ in
Corollary~\ref{lemma2b}. In the following, we will derive a few results that
help to ignore parts negligible when taking this limit. For this purpose, we
say that a real sequence $(a_{n})_{n\in\mathbb{N}}$ is \emph{asymptotically
equivalent} to another real sequence $(b_{n})_{n\in\mathbb{N}}$ in the limit
$n\rightarrow\infty$ if $\lim_{n\rightarrow\infty}\frac{a_{n}}{b_{n}}=1$.

\begin{proposition}
\label{Nprop1} Let $\theta=(\theta_{1},\dots,\theta_{\kappa})\in
\{1,\dots,k\}^{\kappa}$. Consider the sequence
\[
\xi_{n}:=\sum_{\eta\in\Gamma_{\kappa}(n+\kappa-1)}\rho_{\theta_{1}}^{\eta_{1}%
}\cdots\rho_{\theta_{\kappa}}^{\eta_{\kappa}}\quad\mbox{for all }n\in
\mathbb{N}\,.
\]
Then the sequence
\[
\rho(\theta)^{n+1-\kappa}\frac{n^{h^{+}(\theta)-1}}{(h^{+}(\theta)-1)!}%
\prod_{u\in H^{-}(\theta)}\frac{1}{1-\frac{\rho_{\theta_{u}}}{\rho(\theta)}}%
\]
is asymptotically equivalent to $(\xi_{n})_{n\in\mathbb{N}}$ for
$n\rightarrow\infty$.
\end{proposition}

\begin{proof}
For $\ell,m\in\mathbb{N}$ and $\zeta_{1},\dots,\zeta_{\ell}>0$, we introduce
the auxiliary function
\[
\Xi_{\ell}^{m}(\zeta_{1},\dots,\zeta_{\ell}):= \sum_{\eta\in\Gamma_{\ell}(m)
}\zeta_{1}^{\eta_{1}}\cdots\zeta_{\ell}^{\eta_{\ell}} =\sum_{\eta_{1}=0}%
^{m}\sum_{\eta_{2}=0}^{m-\eta_{1}}\dots\sum_{\eta_{\ell}=0}^{m-\eta_{1}%
-\dots-\eta_{\ell-1}}\zeta_{1}^{\eta_{1}}\cdots\zeta_{\ell}^{\eta_{\ell}}\,,
\]
and for $\zeta_{\ell}\not =1$, we can write
\begin{align}
\Xi_{\ell}^{m}(\zeta_{1},\dots,\zeta_{\ell})  &  =\sum_{\eta_{1}=0}^{m}%
\sum_{\eta_{2}=0}^{m-\eta_{1}}\dots\sum_{\eta_{\ell-1}=0}^{m-\eta_{1}%
-\dots-\eta_{\ell-2}}\zeta_{1}^{\eta_{1}}\cdots\zeta_{\ell-1}^{\eta_{\ell-1}%
}\frac{1-\zeta_{\ell}^{m+1-\eta_{1}-\dots-\eta_{\ell-1}}}{1-\zeta_{\ell}%
}\nonumber\\
&  =\frac{1}{1-\zeta_{\ell}}\Xi_{\ell-1}^{m}(\zeta_{1},\dots,\zeta_{\ell
-1})-\frac{\zeta_{\ell}^{m+1}}{1-\zeta_{\ell}}\Xi_{\ell-1}^{m}\left(
\frac{\zeta_{1}}{\zeta_{\ell}},\dots,\frac{\zeta_{\ell-1}}{\zeta_{\ell}%
}\right)  \,. \label{Xi}%
\end{align}
We note that the function $\Xi_{\ell}^{m}$ is symmetric in the sense that
$\Xi_{\ell}^{m}(\zeta_{1},\dots,\zeta_{\ell})=\Xi_{\ell}^{m}(\zeta
_{s(1)},\dots,\zeta_{s(\ell)})$ for all permutations $s$ of $1,\dots,\ell$.
For this reason, the above reformulation of $\Xi_{\ell}^{m}$ into two terms of
the form $\Xi_{\ell-1}^{m}$ can be made as long as not all $\zeta_{i}$,
$i\in\{1,\dots,\ell\}$, are equal to $1$.

We assume without loss of generality that $\theta_{\kappa}=\rho(\theta)$.
Then
\begin{align*}
\xi_{n}  &  =\sum_{\eta\in\Gamma_{\kappa}(n+1-\kappa)}\rho_{\theta_{1}}%
^{\eta_{1}}\cdots\rho_{\theta_{\kappa}}^{\eta_{\kappa}}\\
&  =\sum_{\eta_{1}=0}^{n+1-\kappa}\sum_{\eta_{2}=0}^{n+1-\kappa-\eta_{1}}%
\dots\sum_{\eta_{\kappa-1}=0}^{n+1-\kappa-\eta_{1}-\dots-\eta_{\kappa-2}}%
\rho_{\theta_{1}}^{\eta_{1}}\cdots\rho_{\theta_{\kappa-1}}^{\eta_{\kappa-1}%
}\rho(\theta)^{n+1-\kappa-\eta_{1}-\dots-\eta_{\kappa-1}}\\
&  =\rho(\theta)^{n+1-\kappa}\sum_{\eta_{1}=0}^{n+1-\kappa}\sum_{\eta_{2}%
=0}^{n+1-\kappa-\eta_{1}}\dots\sum_{\eta_{\kappa-1}=0}^{n+1-\kappa-\eta
_{1}-\dots-\eta_{\kappa-2}}\left(  \frac{\rho_{\theta_{1}}}{\rho(\theta
)}\right)  ^{\eta_{1}}\cdots\left(  \frac{\rho_{\theta_{\kappa-1}}}%
{\rho(\theta)}\right)  ^{\eta_{\kappa-1}}.
\end{align*}
Thus we have
\[
\rho(\theta)\xi_{n}=\rho(\theta)^{n+2-\kappa}\Xi_{\kappa-1}^{n+1-\kappa
}\left(  \frac{\rho_{\theta_{1}}}{\rho(\theta)},\dots,\frac{\rho
_{\theta_{\kappa-1}}}{\rho(\theta)}\right)  \quad\mbox{for all }n\in
\mathbb{N}\,.
\]
If $H^{-}(\theta)=\emptyset$, then
\[
\rho(\theta)\xi_{n}=\rho(\theta)^{n+2-\kappa}\Xi_{\kappa-1}^{n+1-\kappa
}(1,\dots,1)\,.
\]
Otherwise, we may assume that $\rho_{\theta_{\kappa-1}}<\rho(\theta)$, and
formula \eqref{Xi} yields%
\begin{align*}
\rho(\theta)\xi_{n}  &  =\rho(\theta)^{n+2-\kappa}\frac{1}{1-\frac
{\rho_{\theta_{\kappa-1}}}{\rho(\theta)}}\Xi_{\kappa-2}^{n+1-\kappa}\left(
\frac{\rho_{\theta_{1}}}{\rho(\theta) },\dots,\frac{\rho_{\theta_{\kappa-2}}%
}{\rho(\theta)}\right) \\
&  \qquad\qquad-\rho_{\theta_{\kappa-1}}^{n+2-\kappa}\frac{1}{1-\frac
{\rho_{\theta_{\kappa-1}}}{\rho(\theta)}}\Xi_{\kappa-2}^{n+1-\kappa}\left(
\frac{\rho_{\theta_{1}}}{\rho_{\theta_{\kappa-1}}},\dots,\frac{\rho
_{\theta_{\kappa-2}}}{\rho_{\theta_{\kappa-1}}}\right)  .
\end{align*}
Using the properties of the function $\Xi_{\ell}^{m}$, we can iteratively
re-write $\rho(\theta)\xi_{n}$ into at most $2^{\kappa-1}$ terms of the form
\begin{equation}
\rho_{\theta_{\gamma(i)}}^{n+2-\kappa}K(i)\Xi_{\beta(i)}^{n+1-\kappa}%
(1,\dots,1)\,,\quad\mbox{where }i\in\{1,\dots,2^{\kappa-1}\}\,,
\label{finalterm}%
\end{equation}
with $\gamma(i)\in\{1,\dots,\kappa\}$, $K(i)\in\mathbb{R}$, and $\beta
(i)\in\{0,\dots,\kappa-1\}$. In all variations of this (non-unique) iterative
procedure, one has a unique term of the form
\begin{equation}
\rho(\theta)^{n+2-\kappa}\Xi_{h^{+}(\theta)-1}^{n+1-\kappa}(1,\dots
,1)\prod_{u\in H^{-}(\theta)}\frac{1}{1-\frac{\rho_{\theta_{u}}}{\rho(\theta
)}}\,, \label{domterm}%
\end{equation}
and we show that this sequence is asymptotically equivalent to $\left(
\rho(\theta)\xi_{n}\right)  _{n\in\mathbb{N}}$.

Firstly, we note that $\frac{m^{\kappa}}{\kappa!}$ is asymptotically
equivalent to $\Xi_{\kappa}^{m}(1,\dots,1)$ for $m\to\infty$. This follows
from the fact that one can show that $\Xi_{\kappa}^{m}(1,\dots,1) =
\#\Gamma_{\kappa+1}(m)$, and we use \eqref{anzgamma}.

In addition, on the way to get to terms of the form \eqref{finalterm}, the
intermediate terms are of the form
\begin{equation}
\rho_{\theta_{\gamma}}^{n+2-\kappa}K\Xi_{\beta}^{n+1-\kappa}\left(  \frac
{\rho_{\theta_{s(1)}}}{\rho_{\theta_{\gamma}}},\dots,\frac{\rho_{\theta
_{s(\beta)}}}{\rho_{\theta_{\gamma}}}\right)  \,, \label{intermterm}%
\end{equation}
where $\gamma\in\{1,\dots,\kappa\}$, $K\in\mathbb{R}$, $\beta\in
\{0,\dots,\kappa-1\}$ and $s:\{1,\dots,\beta\}\rightarrow\{1,\dots,\kappa-1\}$
is injective. This implies that if $\rho_{\theta_{\gamma(i)}}=\rho(\theta)$ in
a final sequence \eqref{finalterm}, then, in the formulation of
\eqref{intermterm}, we have $\rho_{s(j)}=\rho(\theta)$ for all $j\in
\{1,\dots,\beta\}$. It can be seen that in all such terms that do not coincide
with \eqref{domterm}, we have $\beta<h^{+}(\theta)-1$. Hence, the term
\eqref{domterm} yields an asymptotically equivalent sequence, which finishes
the proof.
\end{proof}

In order to analyse both the denominator and numerator from
Corollary~\ref{lemma2b}, we need the following notation and elementary
statements for the diagonal blocks.

\begin{lemma}
[Notation and statements for the diagonal blocks of $Q$]\label{lemma1} For
$i\in\{1,\dots,k\}$, the normed right eigenvector of the Perron--Frobenius
eigenvalue $\rho_{i}$ of $Q_{ii}$ is denoted by $v_{i}$. Since for any
$i\in\{1,\dots,k\}$, the matrix $Q_{ii}$ is eventually positive, the absolute
value of all other eigenvalues is less than some constant $\rho_{i}^{-}%
\in(0,\rho_{i})$, and we denote by $V_{i}^{-}$ the sum of the corresponding
generalised eigenspaces, so that we have the decomposition $\mathbb{R}^{d_{i}%
}=\operatorname{span}(v_{i})\oplus V_{i}^{-}$. In the trivial case $d_{i}=1$,
we have $v_{i}=1$ and $V_{i}^{-}=\{0\}$.

\begin{itemize}
\item[(i)] We define
\[
K_{1}:=\max\big\{1,\max\big\{\|Q_{ij}\|: i>j \big\}\big\}\,.
\]

\item[(ii)] Choose $\gamma$ with
\[
\max\Big\{ \tfrac{\rho_{i}^{-}}{\rho_{i}}:i\in\{1,\dots,k\}\Big\}<\gamma<1
\,.
\]
Then there exists a constant $K_{2}\geq1$ such that for every $i\in
\{1,\dots,k\}$ and $x\in V_{i}^{-}$, we have
\[
\left\Vert Q_{ii}^{n}x\right\Vert \leq K_{2}\rho_{i}^{n}\gamma^{n}%
\|x\|\quad\text{for all }\,n\in\mathbb{N}\,.
\]

\item[(iii)] There exists a constant $K_{3}\ge1$ such that for all
$i\in\{1,\dots,k\}$ and sequences $(x_{n})_{n\in\mathbb{N}}$ in $\mathbb{R}%
^{d}$ with%
\[
x_{n}=z_{n}+w_{n} \quad\text{with }z_{n}\in\operatorname{span}(v_{i})
\mbox{ and }w_{n}\in V_{i}^{-}\,,
\]
the following holds: if for some $\zeta\in(0,1)$ and $K\geq1$, one has
$\left\Vert x_{n}\right\Vert \leq K\zeta^{n}$ for all $n\in\mathbb{N}$, then
\[
\left\Vert z_{n}\right\Vert \leq KK_{3}\zeta^{n} \quad\mbox{and}\quad
\left\Vert w_{n}\right\Vert \leq KK_{3}\zeta^{n} \quad\mbox{for all }
n\in\mathbb{N}\,.
\]

\item[(iv)] Consider an admissible path $\theta\in\mathcal{P}$. We define
$\alpha_{\theta}:=\alpha_{\theta_{\kappa}}\cdots\alpha_{\theta_{1}}$, where
the real numbers $\alpha_{\theta_{u}}$ for all $u\in\{1,\dots,\kappa\}$ are
defined by
\[
\1 =\alpha_{\theta_{\kappa}}v_{\theta_{\kappa}}+w_{\theta_{\kappa}}%
\]
with $w_{\theta_{\kappa}}\in V_{\theta_{\kappa}}^{-}$, and inductively for
$u\in\{\kappa-1,\kappa-2,\dots,1\}$ by
\[
Q_{\theta_{u}\theta_{u+1}}v_{\theta_{u+1}}=\alpha_{\theta_{u}}v_{\theta_{u}%
}+w_{\theta_{u}}%
\]
with $w_{\theta_{u}}\in V_{\theta_{u}}^{-}$. If all submatrices $Q_{\theta
_{u}\theta_{u+1}}$ are scalar, one has, writing $q_{\theta_{u}\theta_{u+1}%
}:=Q_{\theta_{u}\theta_{u+1}}$, that the constants $\alpha_{\theta_{u}}$ are
given by $\alpha_{\theta_{\kappa}}=1$ and $\alpha_{\theta_{u}}=q_{\theta
_{u}\theta_{u+1}}$ for $u\in\{\kappa-1,\kappa-2,\dots,1\}$.

\item[(v)] There exists a $K_{4}\geq1$ such that for all $i\in\{1,\dots,k\}$,
we have
\[
\Vert Q_{ii}^{n}\Vert\leq K_{4}\rho_{i}^{n}\quad\mbox{for all }n\in
\mathbb{N}\,.
\]

\end{itemize}
\end{lemma}

\begin{proof}
(i) and (iv) concern notation and do not need to be proved. For the proof of
(ii) note that for every matrix $Q_{ii}$, Seneta~\cite[Theorem~1.2]{Sene06}
implies that%
\[
Q_{ii}^{n}=\rho_{i}^{n}v_{i}u_{i}^{\top}+\mathcal{O}\left(  ( \rho_{i}^{-})
^{n}\right)  \,,
\]
where $u_{i}^{\top}$ is the positive left eigenvector of $Q_{ii}$ for
$\rho_{i}$ with $u_{i}^{\top}v_{i}=1$. Then it follows for $x\in V_{i}^{-}$
that $v_{i}u_{i}^{\top}x=0$, since otherwise $Q_{ii}^{n}x$ would grow with
$\rho_{i}^{n}$. This implies assertion (ii). Assertion (v) is clear, since the
eigenspace to the maximal real eigenvalue $\rho_{i}$ of $Q_{ii}$ is
one-dimensional (we assumed that the matrix $Q_{ii}$ is eventually positive).
For assertion (iii), the observation below used for the spaces
$X=\operatorname{span}(v_{i})\oplus V_{i}^{-}$ yields a constant
$K_{i}^{\prime}\geq1$ for every $i\in\{1,\dots,k\}$, the maximum of which we
denote by $K_{3}\geq1$. \newline\emph{Observation.} Consider in a
finite-dimensional space $X=Z\oplus W$ a sequence $x_{n}=z_{n}+w_{n}$ with
\thinspace$z_{n}\in Z,w_{n}\in W$, and $\left\Vert x_{n}\right\Vert \leq
K\zeta^{n}$ for some $\zeta\in(0,1)$ and $K\geq1$. Then there exists a
constant $K^{\prime}\geq1$ such that $\left\Vert z_{n}\right\Vert \leq
K^{\prime}K\zeta^{n}$ and $\left\Vert w_{n}\right\Vert \leq K^{\prime}%
K\zeta^{n}$ for all $n\in\mathbb{N}$. \newline\emph{Proof of the observation.}
In fact, for a norm such that $\left\Vert x\right\Vert ^{\prime}=\left\Vert
z\right\Vert ^{\prime}+\left\Vert w\right\Vert ^{\prime}$ for $x\in X$ with
$z\in Z,w\in W$, one has $\left\Vert z_{n}\right\Vert ^{\prime}\leq$
$\left\Vert z_{n}\right\Vert ^{\prime}+$ $\left\Vert w_{n}\right\Vert
^{\prime}=$ $\left\Vert x_{n}\right\Vert ^{\prime}\leq K\gamma^{n}$,
analogously for $w_{n}$. This result remains true for every norm $\left\Vert
\cdot\right\Vert $, since all norms on finite-dimensional spaces are
equivalent. In fact, $c^{-1}\left\Vert x\right\Vert ^{\prime}\leq\left\Vert
x\right\Vert \leq c\left\Vert x\right\Vert ^{\prime}$ for some constant $c>0$,
hence $\left\Vert z_{n}\right\Vert \leq c\left\Vert z_{n}\right\Vert ^{\prime
}\leq c\left\Vert x_{n}\right\Vert ^{\prime}\leq c^{2}\left\Vert
x_{n}\right\Vert \leq K^{\prime}K\zeta^{n}$ with $K^{\prime}:=c^{2}$.
\end{proof}

In the following proposition, we aim at understanding the asymptotic growth of
sequences of the form $\pi_{\theta_{1}}Q(\theta,n)\1 $ which occur in
the denominator in Corollary~\ref{lemma2b}.

\begin{proposition}
\label{prop1} Consider a matrix $Q$ of the form \eqref{P_Q} and an admissible
path $\theta=(\theta_{1},\dots,\theta_{\kappa})\in\mathcal{P}$ and suppose
that for all $u\in\{1,\dots,\kappa\}$ with $\rho_{\theta_{u}}<\rho(\theta)$,
the diagonal term $Q_{\theta_{u}\theta_{u}}$ is scalar.

\begin{itemize}
\item[(i)] If $\pi_{\theta_{1}}v_{\theta_{1}}\not =0$ and $\alpha_{\theta
}\not =0$, then the sequence $\pi_{\theta_{1}}Q(\theta,n)\1 $ is
asymptotically equivalent in the limit $n\rightarrow\infty$ to
\begin{equation}
\pi_{\theta_{1}}v_{\theta_{1}}\alpha_{\theta}\sum_{\eta\in\Gamma_{\kappa
}(n+1-\kappa)}\rho_{\theta_{\kappa}}^{\eta_{\kappa}}\cdots\rho_{\theta_{1}%
}^{\eta_{1}}\,, \label{K_4}%
\end{equation}
and hence, due to Proposition~\ref{Nprop1}, also to the sequence
\begin{equation}
\alpha_{\theta}\pi_{\theta_{1}}v_{\theta_{1}}\rho(\theta)^{n+1-\kappa}%
\frac{n^{h^{+}(\theta)-1}}{(h^{+}(\theta)-1)!}\prod_{u\in H^{-}(\theta)}%
\frac{1}{1-\frac{\rho_{\theta_{u}}}{\rho(\theta)}}\,. \label{asequ45}%
\end{equation}

\item[(ii)] If $\pi_{\theta_{1}}v_{\theta_{1}}=0$ or $\alpha_{\theta}=0$, then
the exponential growth of $\pi_{\theta_{1}}Q(\theta,n)\1 $ in the limit
$n\rightarrow\infty$ is equal to or less than $\rho(\theta)^{n+1-\kappa
}n^{h^{+}(\theta)-2}$.
\end{itemize}
\end{proposition}

\begin{proof}
Due to \eqref{P_Q_theta}, we have
\begin{equation}
\pi_{\theta_{1}}Q(\theta,n)\1 =\pi_{\theta_{1}}\,\sum_{\eta\in
\Gamma_{\kappa}(n+1-\kappa)}Q_{\theta_{1}\theta_{1}}^{\eta_{1}}Q_{\theta
_{1}\theta_{2}}Q_{\theta_{2}\theta_{2}}^{\eta_{2}}Q_{\theta_{2}\theta_{3}%
}\cdots Q_{\theta_{\kappa-1}\theta_{\kappa}}Q_{\theta_{\kappa}\theta_{\kappa}%
}^{\eta_{\kappa}}\1 \,. \label{N01}%
\end{equation}
Since we assume that for $\rho_{i}<\rho(\theta)$, the diagonal term $Q_{ii}$
is scalar, it follows that for $u\in H^{-}(\theta)$, the decomposition
$Q_{\theta_{u}\theta_{u+1}}v_{\theta_{u+1}}=\alpha_{\theta_{u}}v_{\theta_{u}%
}+w_{\theta_{u}}$ from Lemma~\ref{lemma1}~(iv) is scalar, and hence using
$v_{\theta_{u}}=1$ and $w_{\theta_{u}}=0$, it is of the form%
\begin{equation}
Q_{\theta_{u}\theta_{u+1}}v_{\theta_{u+1}}=\alpha_{\theta_{u}}\,. \label{rel1}%
\end{equation}
Consider the first iterative step from Lemma~\ref{lemma1}~(iv)
\begin{equation}
\1 =\alpha_{\theta_{\kappa}}v_{\theta_{\kappa}}+w_{\theta_{\kappa}}
\label{N0}%
\end{equation}
with $w_{\theta_{\kappa}}\in V_{\theta_{\kappa}}^{-}$. Using that
$\rho_{\theta_{\kappa}}$ is an eigenvalue of $Q_{\theta_{\kappa}\theta
_{\kappa}}$ with eigenvector $v_{\theta_{\kappa}}$, we get
\begin{equation}
Q_{\theta_{\kappa}\theta_{\kappa}}^{\eta_{\kappa}}\1 =\alpha
_{\theta_{\kappa}}\rho_{\theta_{\kappa}}^{\eta_{\kappa}}v_{\theta_{\kappa}%
}+Q_{\theta_{\kappa}\theta_{\kappa}}^{\eta_{\kappa}}w_{\theta_{\kappa}}.
\label{N2}%
\end{equation}
In the next step, we decompose%
\begin{equation}
Q_{\theta_{\kappa-1}\theta_{\kappa}}v_{\theta_{\kappa}}=\alpha_{\theta
_{\kappa-1}}v_{\theta_{\kappa-1}}+w_{\theta_{\kappa-1}}\text{ } \label{N3}%
\end{equation}
with $\alpha_{\theta_{\kappa-1}}\in\mathbb{R}$ and $w_{\theta_{\kappa-1}}\in
V_{\theta_{\kappa-1}}^{-}$. Hence, Lemma~\ref{lemma1}~(ii) implies
\begin{equation}
\big\|Q_{\theta_{\kappa-1}\theta_{\kappa-1}}^{\eta_{\kappa-1}}w_{\theta
_{\kappa-1}}\big\|\leq K_{2}\rho_{\theta_{\kappa-1}}^{\eta_{\kappa-1}}%
\gamma^{\eta_{\kappa-1}}\Vert w_{\theta_{\kappa-1}}\Vert\,, \label{N3a}%
\end{equation}
and we decompose%
\begin{equation}
Q_{\theta_{\kappa-1}\theta_{\kappa}}Q_{\theta_{\kappa}\theta_{\kappa}}%
^{\eta_{\kappa}}w_{\theta_{\kappa}}=\beta_{\theta_{\kappa-1}}^{(\eta_{\kappa
})}v_{\theta_{\kappa-1}}+w_{\theta_{\kappa-1}}^{(\eta_{\kappa})} \label{N4}%
\end{equation}
with $\beta_{\theta_{\kappa-1}}^{(\eta_{\kappa})}\in\mathbb{R}$ and
$w_{\theta_{\kappa-1}}^{(\eta_{\kappa})}\in V_{\theta_{\kappa-1}}^{-}$. Due to
Lemma~\ref{lemma1}~(i),(ii), the left hand side of \eqref{N4} satisfies
\[
\left\Vert Q_{\theta_{\kappa-1}\theta_{\kappa}}Q_{\theta_{\kappa}%
\theta_{\kappa}}^{\eta_{\kappa}}w_{\theta_{\kappa}}\right\Vert \leq
K_{1}\left\Vert Q_{\theta_{\kappa}\theta_{\kappa}}^{\eta_{\kappa}}%
w_{\theta_{\kappa}}\right\Vert \leq K_{1}K_{2}\rho_{\theta_{\kappa}}%
^{\eta_{\kappa}}\gamma^{\eta_{\kappa}}\left\Vert w_{\theta_{\kappa}%
}\right\Vert .
\]
This implies for the right hand side of \eqref{N4} by Lemma~\ref{lemma1}~(iii)
that%
\begin{equation}
\big\|w_{\theta_{\kappa-1}}^{(\eta_{\kappa})}\big\|\leq K_{1}K_{2}K_{3}%
\rho_{\theta_{\kappa}}^{\eta_{\kappa}}\gamma^{\eta_{\kappa}}\left\Vert
w_{\theta_{\kappa}}\right\Vert \label{N5}%
\end{equation}
and%
\begin{equation}
\big|\beta_{\theta_{\kappa-1}}^{(\eta_{\kappa})}\big|=\big\|\beta
_{\theta_{\kappa-1}}^{(\eta_{\kappa})}v_{_{\theta_{\kappa-1}}}\big\|\leq
K_{1}K_{2}K_{3}\rho_{\theta_{\kappa}}^{\eta_{\kappa}}\gamma^{\eta_{\kappa}%
}\left\Vert w_{\theta_{\kappa}}\right\Vert . \label{N6}%
\end{equation}
Together this yields%
\begin{eqnarray*}
& & \!\!\!Q_{\theta_{\kappa-1}\theta_{\kappa-1}}^{\eta_{\kappa-1}}%
Q_{\theta_{\kappa-1}\theta_{\kappa}}Q_{\theta_{\kappa}\theta_{\kappa}}%
^{\eta_{\kappa}}\1 \\
&  \overset{\eqref{N2}}{=}&\!\!\!Q_{\theta_{\kappa-1}\theta_{\kappa-1}}%
^{\eta_{\kappa-1}}Q_{\theta_{\kappa-1}\theta_{\kappa}}\left(  \alpha
_{\theta_{\kappa}}\rho_{\theta_{\kappa}}^{\eta_{\kappa}}v_{_{\theta_{\kappa}}%
}+Q_{\theta_{\kappa}\theta_{\kappa}}^{\eta_{\kappa}}w_{\theta_{\kappa}}\right)
\\
&  =&\!\!\!\alpha_{\theta_{\kappa}}\rho_{\theta_{\kappa}}^{\eta_{\kappa}%
}Q_{\theta_{\kappa-1}\theta_{\kappa-1}}^{\eta_{\kappa-1}}Q_{\theta_{\kappa
-1}\theta_{\kappa}}v_{\theta_{\kappa}}+Q_{\theta_{\kappa-1}\theta_{\kappa-1}%
}^{\eta_{\kappa-1}}Q_{\theta_{\kappa-1}\theta_{\kappa}}Q_{\theta_{\kappa
}\theta_{\kappa}}^{\eta_{\kappa}}w_{\theta_{\kappa}}\\
&  \overset{\eqref{N3},\eqref{N4}}{=}& \!\!\!\alpha_{\theta_{\kappa}}%
\rho_{\theta_{\kappa}}^{\eta_{\kappa}}Q_{\theta_{\kappa-1}\theta_{\kappa-1}%
}^{\eta_{\kappa-1}}\left(  \alpha_{\theta_{\kappa-1}}v_{\theta_{\kappa-1}%
}+w_{\theta_{\kappa-1}}\right) \\
& & \!\!\!+Q_{\theta_{\kappa-1}\theta_{\kappa-1}}^{\eta_{\kappa-1}}\left(
\beta_{\theta_{\kappa-1}}^{(\eta_{\kappa})}v_{_{\theta_{\kappa-1}}}%
+w_{\theta_{\kappa-1}}^{(\eta_{\kappa})}\right) \\
&  =&\!\!\!\alpha_{\theta_{\kappa}}\alpha_{\theta_{\kappa-1}}\rho
_{\theta_{\kappa}}^{\eta_{\kappa}}Q_{\theta_{\kappa-1}\theta_{\kappa-1}}%
^{\eta_{\kappa-1}}v_{\theta_{\kappa-1}}+\alpha_{\theta_{\kappa}}\rho
_{\theta_{\kappa}}^{\eta_{\kappa}}Q_{\theta_{\kappa-1}\theta_{\kappa-1}}%
^{\eta_{\kappa-1}}w_{\theta_{\kappa-1}}\\
& & \!\!\!+Q_{\theta_{\kappa-1}\theta_{\kappa-1}}^{\eta_{\kappa-1}}%
\beta_{\theta_{\kappa-1}}^{(\eta_{\kappa})}v_{\theta_{\kappa-1}}%
+Q_{\theta_{\kappa-1}\theta_{\kappa-1}}^{\eta_{\kappa-1}}w_{\theta_{\kappa-1}%
}^{(\eta_{\kappa})}\\
&  =&\!\!\!\alpha_{\theta_{\kappa}}\alpha_{\theta_{\kappa-1}}\rho
_{\theta_{\kappa}}^{\eta_{\kappa}}\rho_{\theta_{\kappa-1}}^{\eta_{\kappa-1}%
}v_{\theta_{\kappa-1}}+\alpha_{\theta_{\kappa}}\rho_{\theta_{\kappa}}%
^{\eta_{\kappa}}Q_{\theta_{\kappa-1}\theta_{\kappa-1}}^{\eta_{\kappa-1}%
}w_{\theta_{\kappa-1}}\\
& & \!\!\!+\beta_{\theta_{\kappa-1}}^{(\eta_{\kappa})}\rho_{\theta_{\kappa-1}%
}^{\eta_{\kappa-1}}v_{\theta_{\kappa-1}}+Q_{\theta_{\kappa-1}\theta_{\kappa
-1}}^{\eta_{\kappa-1}}w_{\theta_{\kappa-1}}^{(\eta_{\kappa})}.
\end{eqnarray*}
The last three summands satisfy the estimates
\begin{eqnarray*}
\big\|\alpha_{\theta_{\kappa}}\rho_{\theta_{\kappa}}^{\eta_{\kappa}}%
Q_{\theta_{\kappa-1}\theta_{\kappa-1}}^{\eta_{\kappa-1}}w_{\theta_{\kappa-1}%
}\big\| \!\!\!  &  \leq &\!\!\!\left\vert \alpha_{\theta_{\kappa}}\right\vert
\rho_{\theta_{\kappa}}^{\eta_{\kappa}}\big\|Q_{\theta_{\kappa-1}\theta
_{\kappa-1}}^{\eta_{\kappa-1}}w_{_{\theta_{\kappa-1}}}\big\|\\
\!\!\!  &  \overset{\eqref{N3a}}{\leq} &\!\!\!K_{2}\left\vert \alpha
_{\theta_{\kappa}}\right\vert \rho_{\theta_{\kappa}}^{\eta_{\kappa}}%
\rho_{\theta_{\kappa-1}}^{\eta_{\kappa-1}}\gamma^{\eta_{\kappa-1}}\Vert
w_{\theta_{\kappa-1}}\Vert
\end{eqnarray*}
and
\[
\big\|\beta_{\theta_{\kappa-1}}^{(\eta_{\kappa})}\rho_{\theta_{\kappa-1}%
}^{\eta_{\kappa-1}}v_{\theta_{\kappa-1}}\big\|\leq\big|\beta_{\theta
_{\kappa-1}}^{(\eta_{\kappa})}\big|\rho_{\theta_{\kappa-1}}^{\eta_{\kappa-1}%
}\Vert v_{\theta_{\kappa-1}}\Vert\overset{\eqref{N6}}{\leq}K_{1}K_{2}K_{3}%
\rho_{\theta_{\kappa}}^{\eta_{\kappa}}\rho_{\theta_{\kappa-1}}^{\eta
_{\kappa-1}}\gamma^{\eta_{\kappa}}\left\Vert w_{\theta_{\kappa}}\right\Vert
\]
and
\begin{eqnarray*}
\big\|Q_{\theta_{\kappa-1}\theta_{\kappa-1}}^{\eta_{\kappa-1}}w_{\theta
_{\kappa-1}}^{(\eta_{\kappa})}\big\|\!\!\!  &  \leq &\!\!\!\big\|Q_{\theta
_{\kappa-1}\theta_{\kappa-1}}^{\eta_{\kappa-1}}\big\|\big\|w_{\theta
_{\kappa-1}}^{(\eta_{\kappa})}\big\|\overset{\text{Lemma~\ref{lemma1}~(v)}%
}{\leq} K_{4}\rho_{\theta_{\kappa-1}}^{\eta_{\kappa-1}}\big\|w_{\theta
_{\kappa-1}}^{(\eta_{\kappa})}\big\|\\
\!\!\!  &  \overset{\eqref{N5}}{\leq} &\!\!\!K_{1}K_{2}K_{3}K_{4}\rho
_{\theta_{\kappa}}^{\eta_{\kappa}}\rho_{\theta_{\kappa-1}}^{\eta_{\kappa-1}%
}\gamma^{\eta_{\kappa}}\left\Vert w_{\theta_{\kappa}}\right\Vert \,.
\end{eqnarray*}
If $\kappa\not \in H^{+}(\theta)$, then by \eqref{rel1}, it follows that
$w_{\theta_{\kappa}}=0$, and hence, of the last three summands above, only%
\[
\alpha_{\theta_{\kappa}}\rho_{\theta_{\kappa}}^{\eta_{\kappa}}Q_{\theta
_{\kappa-1}\theta_{\kappa-1}}^{\eta_{\kappa-1}}w_{\theta_{\kappa-1}}%
\]
can be different from $0$. If $\kappa-1\not \in H^{+}(\theta)$, then by
\eqref{rel1}, it follows that $w_{\theta_{\kappa-1}}=0$, and hence, this
summand vanishes. Together with the estimates derived above, it follows that
each of the additional three summands vanishes, if both $\kappa,\kappa
-1\not \in H^{+}(\theta)$, and the norm of each of the additional summands can
be estimated by a constant multiplied with
\begin{align*}
\rho_{\theta_{\kappa}}^{\eta_{\kappa}}\rho_{\theta_{\kappa-1}}^{\eta
_{\kappa-1}}\gamma^{\eta_{\kappa}}\qquad\text{ if }\kappa &  \in H^{+}%
(\theta)\text{ and }\kappa-1\not \in H^{+}(\theta)\,,\\
\rho_{\theta_{\kappa}}^{\eta_{\kappa}}\rho_{\theta_{\kappa-1}}^{\eta
_{\kappa-1}}\gamma^{\eta_{\kappa-1}}\qquad\text{ if }\kappa &  \not \in
H^{+}(\theta)\text{ and }\kappa-1\in H^{+}(\theta)\,,\\
\rho_{\theta_{\kappa}}^{\eta_{\kappa}}\rho_{\theta_{\kappa-1}}^{\eta
_{\kappa-1}}\gamma^{\eta_{\kappa}}\text{ or }\rho_{\theta_{\kappa}}%
^{\eta_{\kappa}}\rho_{\theta_{\kappa-1}}^{\eta_{\kappa-1}}\gamma^{\eta
_{\kappa-1}}\qquad\text{ if }\kappa &  \in H^{+}(\theta)\text{ and }%
\kappa-1\in H^{+}(\theta)\,.
\end{align*}
After $\kappa$ decomposition steps, we arrive at the following result: any
term in the sum in \eqref{N01} (i.e.~for a fixed $\eta\in\Gamma_{\kappa
}(n+1-\kappa)$) is equal to the sum of%
\begin{equation}
\pi_{\theta_{1}}\alpha_{\theta}\rho_{\theta_{\kappa}}^{\eta_{\kappa}}%
\cdots\rho_{\theta_{1}}^{\eta_{1}}v_{\theta_{1}}=\alpha_{\theta}\pi
_{\theta_{1}}v_{\theta_{1}}\rho_{\theta_{\kappa}}^{\eta_{\kappa}}\cdots
\rho_{\theta_{1}}^{\eta_{1}} \label{Ndominant_denominator}%
\end{equation}
and up to $2^{h^{+}(\theta)}-1$ summands that contain in addition to the
factor $\rho_{\theta_{\kappa}}^{\eta_{\kappa}}\cdots\rho_{\theta_{1}}%
^{\eta_{1}}$ some factor $\gamma^{\eta_{u}}$ with $u\in H^{+}(\theta)$, hence,
$\gamma\rho_{\theta_{u}}< \rho(\theta)$. It follows from
Proposition~\ref{Nprop1} that these additional terms have in the limit
$n\rightarrow\infty$ exponential growth equal to or less than $\rho
(\theta)^{n+1-\kappa}n^{h^{+}(\theta)-2}$. Hence for $\pi_{\theta_{1}}%
\alpha_{\theta}v_{\theta_{1}}\not =0$, the sequences in
\eqref{Ndominant_denominator} and \eqref{N01} are asymptotically equivalent,
and assertion (i) holds. If $\pi_{\theta_{1}}\alpha_{\theta}v_{\theta_{1}}=0$
the terms in \eqref{Ndominant_denominator} vanish and the estimate for the
exponential growth of the other summands implies that (ii) holds.
\end{proof}

\begin{remark}
\label{genericity_alpha}If for an admissible path $\theta\in\mathcal{P}$, all
submatrices $Q_{\theta_{u}\theta_{u+1}}$ are scalar, the assumption
$\alpha_{\theta}=\alpha_{\theta_{\kappa}}\cdots\alpha_{\theta_{1}}\not =0$ in
Proposition~\ref{prop1} holds. In the general case, it is generically
satisfied: for matrices $Q$ of the form \eqref{P_Q}, recall that $v_{i}$
denotes the Perron--Frobenius eigenvector of the diagonal block $Q_{ii}$ for
$i\in\{1,\dots,k\}$. Then the set of matrices $Q$ such that the decomposition
$\1 =\alpha_{k}v_{k}+w_{k}$ with $\alpha_{k}\in\mathbb{R}$ and
$w_{k}\in V_{k}^{-}$ satisfies $\alpha_{k}\not =0$, is open and dense.
Similarly, for all $i,j\in\{1,\dots,k\}$ with $i>j$, the set of matrices such
that the decomposition
\[
Q_{ij}v_{j}=\alpha_{ij}v_{i}+w_{ij}\text{ with }\alpha_{ij}\in\mathbb{R}\text{
and }w_{ij}\in V_{i}^{-},
\]
satisfies $\alpha_{ij}\not =0$, is open and dense. This implies (cf.~Lemma
\ref{lemma1}~(iv)) that the set of matrices $Q$ such that for every admissible
path $\theta\in\mathcal{P}$, all numbers $\alpha_{\theta_{\kappa}}%
,\dots,\alpha_{\theta_{1}}$ are nonzero, is open and dense.
\end{remark}

The following two examples further illustrate the assumptions of Proposition
\ref{prop1}.

\begin{example}
We demonstrate now that the assumption $\alpha_{\theta}=\alpha_{\theta
_{\kappa}}\cdots\alpha_{\theta_{1}}\not =0$ in Proposition~\ref{prop1} (i) is
not satisfied in general. Consider the matrix
\[
Q=%
\begin{pmatrix}
Q_{11} & 0\\
Q_{21} & Q_{22}%
\end{pmatrix}
\,,\quad\mbox{where }Q_{11},Q_{21},Q_{22}\in\mathbb{R}^{n\times n}.
\]
Suppose that $Q_{11}$ and $Q_{22}$ are eventually positive with
Perron--Frobenius eigenvalues $\rho_{1}$ and $\rho_{2}$ with normalised
positive right eigenvectors $v_{1}$ and $v_{2}$, respectively. Furthermore,
let $V_{2}^{-}\subset\mathbb{R}^{n}$ be the subspace spanned by the
eigenvectors of $Q_{22}$ corresponding to the eigenvalues with smaller
magnitude than $\rho_{2}$, as defined in Lemma~\ref{lemma1}. Suppose that
$Q_{21}$ is a matrix and $Q_{21}v_{1}\in V_{2}^{-}$. For $\theta=(\theta
_{1},\theta_{2})=(2,1)$ and $\kappa=2$, one has $v_{\theta_{\kappa}}=v_{1}$
and $v_{\theta_{\kappa-1}}=v_{2}$. Then it follows that $\alpha_{\theta
_{\kappa-1}}=\alpha_{2}=0$, since due to \eqref{N3}, we have
\[
Q_{21}v_{1}=0\cdot v_{2}+w_{2}\,,\quad\text{ with }w_{2}\in V_{2}^{-}\,.
\]

\end{example}

We now take a closer look at the assumption requiring certain diagonal terms
to be scalar.

\begin{example}
\label{counterexample1}In Proposition~\ref{prop1}, the assumption that the
diagonal term $Q_{\theta_{u}\theta_{u}}$ is scalar for all $u\in
\{1,\dots,\kappa\}$ with $\rho_{\theta_{u}}<\rho(\theta)$, is necessary and
cannot be omitted in general. Let $d=3$ and consider a matrix of the form
\[
Q=%
\begin{pmatrix}
\rho_{1} & 0\\
Q_{21} & Q_{22}%
\end{pmatrix}
\,,
\]
where $Q_{22}\in\mathbb{R}^{2\times2}$ is eventually positive with the simple
eigenvalues $\rho_{2}>\rho_{2}^{-}>0$ and $\rho_{1}>\rho_{2}$. Here $\rho
_{1}=\rho(\theta)$ and $Q_{11}= (\rho_{1})$, and we assume that $Q_{21}%
=(q_{21},q_{31})^{\top}$ has positive entries, and hence, the path
$\theta=(\theta_{1},\theta_{2})=(2,1)$ with $\kappa=2$ is admissible. Then
$H^{+}(\theta)=\{1\}$, $H^{-}(\theta)=\{2\}$ and%
\[
\theta_{\kappa}=\theta_{2}=1 \quad\text{and} \quad\theta_{\kappa-1}=\theta
_{1}=2\,.
\]
The eigenvalue of $Q_{11}$ is $\rho_{1}=\rho(\theta)$ with normed eigenvector
$1\in\mathbb{R}$. The decomposition~\eqref{N0} reads as $\1 %
=\alpha_{\theta_{\kappa}}v_{\theta_{\kappa}}+w_{\theta_{\kappa}}$ in
$\mathbb{R}$ with $\theta_{\kappa}=\theta_{2}=1$, and we get $\alpha
_{\theta_{\kappa}}=\alpha_{1}=\rho_{1}$, $v_{\theta_{\kappa}}=v_{1}=1$ and
$w_{\theta_{\kappa}}=w_{1}=0$. The matrix $Q_{22}$ has normalised eigenvectors
$v_{2}$ for $\rho_{2}$ and $v_{2}^{-}$ for $\rho_{2}^{-}$. Thus the subspace
$V_{2}^{-}$ is spanned by the eigenvector $v_{2}^{-}$, and hence, the
decomposition \eqref{N3} in $\mathbb{R}^{2}$ has the form%
\[
Q_{21}1=
\begin{pmatrix}
q_{21}\\
q_{31}%
\end{pmatrix}
=\alpha_{2}v_{2}+w_{2}=\alpha_{2}v_{2}+cv_{2}^{-}\quad\text{with }cv_{2}%
^{-}\in V_{2}^{-}\text{ for some }c\in\mathbb{R}.
\]
We further assume that $w_{2}\not =0$, and hence $c\not =0$. The decomposition
\eqref{N4} is trivial, since $w_{\theta_{\kappa}}=w_{1}=0$, showing that
$\beta_{\theta_{\kappa-1}}^{(\eta_{\kappa})}=0$ and $w_{\theta_{\kappa-1}%
}^{(\eta_{\kappa})}=0$. Together, this yields the formula%
\[
Q_{22}^{\eta_{2}}Q_{21}Q_{11}^{\eta_{1}}1=\rho_{1}\alpha_{2}\rho_{1}^{\eta
_{2}}\rho_{2}^{\eta_{1}}v_{2}+\rho_{1}\rho_{1}^{\eta_{2}}Q_{22}^{\eta_{1}%
}w_{2}=\rho_{1}\alpha_{2}\rho_{1}^{\eta_{2}}\rho_{2}^{\eta_{1}}v_{2}+\rho
_{1}\rho_{1}^{\eta_{2}}c\left(  \rho_{2}^{-}\right)  ^{\eta_{1}}v_{2}^{-}.
\]
Now we sum the right hand side over all $\eta=(\eta_{1},\eta_{2})\in
\mathbb{N}_{0}^{2}$ with $\eta_{1}+\eta_{2}=n+1-\kappa=n-1$. This gives for
the first summand of the right hand side
\begin{align*}
&  \sum\limits_{\eta\in\Gamma_{2}(n-1)}\rho(\theta)\alpha_{2}\rho^{\eta_{2}%
}\rho_{2}^{\eta_{1}}v_{2} =\rho(\theta)\alpha_{2}\sum_{\eta_{1}=0}^{n-1}%
\rho(\theta) ^{n-1-\eta_{1}}\rho_{2}^{\eta_{1}}v_{2}\\
&  =\rho(\theta)\alpha_{2}\rho(\theta)^{n-1}\sum_{\eta_{1}=0}^{n-1}\left(
\frac{\rho_{2}}{\rho(\theta)}\right)  ^{\eta_{1}}v_{2} =\alpha_{2}\rho
(\theta)^{n}\frac{1-\left(  \frac{\rho_{2}}{\rho(\theta)}\right)  ^{n}%
}{1-\frac{\rho_{2}}{\rho(\theta)}}v_{2}\,.
\end{align*}
Similarly, the second summand yields
\[
\sum\limits_{\eta\in\Gamma_{2}(n-1)}\rho(\theta)^{\eta_{2}+1}c\left(  \rho
_{2}^{-}\right)  ^{\eta_{1}}v_{2}^{-}=c\rho(\theta)^{n}\sum_{\eta_{1}=0}%
^{n-1}\left(  \frac{\rho_{2}^{-}}{\rho(\theta)}\right)  ^{\eta_{1}}v_{2}%
^{-}=c\rho(\theta)^{n}\frac{1-\left(  \frac{\rho_{2}^{-}}{\rho(\theta
)}\right)  ^{n}}{1-\frac{\rho_{2}^{-}}{\rho(\theta)}}v_{2}^{-}\,.
\]
Thus, we get with $\pi_{1}=\pi_{\theta_{2}}$ and $(\pi_{2,1},\pi_{2,2}%
)=\pi_{2}=\pi_{\theta_{1}}$ that
\begin{align*}
&  (n+1)\pi_{\theta_{1}}Q(\theta,n)\1  =(n+1)(\pi_{2,1},\pi_{2,2}%
)\sum\limits_{\eta\in\Gamma_{2}(n-1)}Q_{22}^{\eta_{2}}Q_{21}Q_{11}^{\eta_{1}%
}1\\
&  =(n+1)(\pi_{2,1},\pi_{2,2})v_{2}\alpha_{2}\rho(\theta)^{n}\frac{1-\left(
\frac{\rho_{2}}{\rho(\theta)}\right)  ^{n}}{1-\frac{\rho_{2}^{-}}{\rho
(\theta)}}\\
&  \quad+(n+1)(\pi_{2,1},\pi_{2,2})v_{2}^{-}c\rho(\theta)^{n}\frac{1-\left(
\frac{\rho_{2}^{-}}{\rho(\theta)}\right)  ^{n}}{1-\frac{\rho_{2}^{-}}%
{\rho(\theta)}}\,.
\end{align*}
Since $\Big(\frac{\rho_{2}^{-}}{\rho(\theta)}\Big)  ^{n}\rightarrow0$ for
$n\rightarrow\infty$, one concludes that for $(\pi_{2,1},\pi_{2,2} )v_{2}%
^{-}\not =0$, the first summand is asymptotically equivalent to
\eqref{asequ45}, and the second summand is not asymptotically equivalent to
$0$, so the assertion of Proposition~\ref{prop1} does not hold in this case.
\end{example}

\begin{remark}
Proposition~\ref{prop1} sharpens Theorem~9.4 in the survey Schneider
\cite{Schn86} which asserts that for any matrix of the form \eqref{P_Q}, the
submatrix $Q_{ij}^{(n)}$ has exponential growth rate $s(i,j)^{n}n^{d(i,j)}$,
where $s(i,j)$ is the maximum of $\rho_{k}$ which lie on an admissible path
from $i$ to $j$ and $d(i,j)+1$ is the number of $k$ with $\rho_{k}=s(i,j)$. In
our terminology, $d(i,j)+1=h^{+}(\theta)$, hence this theorem implies that the
exponential growth rate is given by $\rho_{\max}^{n}n^{\max_{\theta}%
h^{+}(\theta)-1}$, where the maximum is taken over all admissible paths
$\theta$ from $i$~to~$j$.
\end{remark}

Proposition~\ref{prop1} can be immediately applied to the summands in the
denominators of Corollary~\ref{lemma2b}. In the following, we show that also
the asymptotic behaviour of the terms in the numerator of
Corollary~\ref{lemma2b} can be understood via Proposition~\ref{prop1}, but for
this purpose, we need to replace the matrix $Q$ by the following matrix. For
$\ell\in\{1,\dots,k\}$, consider
\[
\hat{Q}^{(\ell)}:=%
\begin{pmatrix}
Q_{11} &  &  &  &  & 0\\
\vdots & \ddots &  &  &  & \\
Q_{\ell1} & \hdots & Q_{\ell\ell} &  &  & \\
&  & \operatorname{Id}_{d_{\ell}} & Q_{\ell\ell} &  & \\
&  &  & \vdots & \ddots & \\
0 &  &  & Q_{k\ell} & \hdots & Q_{kk}%
\end{pmatrix}
\,.
\]
We note that any sub-matrix $Q_{ij}$ with $i>\ell$ and $j<\ell$ does not
appear in this matrix. The matrix $\hat{Q}^{(\ell)}$ is also of the form
\eqref{P_Q} with $k+1$ blocks, one more block than the matrix $Q$. Denote the
elements of $\Gamma_{\kappa+1}(n+1-\kappa)$ by $\hat{\eta}=(\eta_{1}%
,\dots,\eta_{\ell},\hat{\eta}_{\ell},\eta_{\ell+1},\dots,\eta_{\kappa})$, and
define for $\theta=(\theta_{1},\dots,\theta_{\kappa})\in\mathcal{P}^{(\ell)}%
$,
\[
\hat{Q}^{(\ell)}(\theta,n)=\sum_{\hat{\eta}\in\Gamma_{\kappa+1}(n+1-\kappa
)}Q_{\theta_{1}\theta_{1}}^{\eta_{1}}Q_{\theta_{1}\theta_{2}}\cdots
Q_{\theta_{u-1}\ell}Q_{\ell\ell}^{\eta_{\ell}}Q_{\ell\ell}^{\hat{\eta}_{\ell}%
}Q_{\ell\theta_{u+1}}\cdots Q_{\theta_{\kappa-1}\theta_{\kappa}}%
Q_{\theta_{\kappa}\theta_{\kappa}}^{\eta_{\kappa}}.
\]

We now aim at understanding the asymptotic growth of sequences of the form
$\pi_{\theta_{1}}\hat{Q}^{(\ell)}(\theta,n)\1 =\pi_{\theta_{1}}%
\sum_{r=0}^{n}Q(\underline{\theta^{\ell} },r)Q(\overline{\theta^{\ell}%
},n-r)\1 $, which occur in the numerators in Corollary~\ref{lemma2b}.
The main idea is to apply Proposition~\ref{prop1}, where $Q(\theta,n)$ is
replaced by $\hat{Q}^{(\ell)}(\theta,n)$.

\begin{proposition}
\label{Nproposition4.24} Consider a matrix $Q$ of the form \eqref{P_Q}, a
number $\ell\in\{1,\dots,k\}$ and an admissible path $\theta=(\theta_{1}%
,\dots,\theta_{\kappa})\in\mathcal{P}^{(\ell)}$. Furthermore, suppose that for
all $u\in\{1,\dots,\kappa\}$ with $\rho_{\theta_{u}}<\rho(\theta)$, the
diagonal term $Q_{\theta_{u}\theta_{u}}$ is scalar. Then for every $\theta
\in\mathcal{P}^{(\ell)}$, we have
\begin{equation}
\hat{Q}^{(\ell)}(\theta,n)=\sum_{r=0}^{n}Q(\underline{\theta^{\ell}%
},r)Q(\overline{\theta^{\ell}},n-r)\,, \label{4.19a}%
\end{equation}
and the following two statements hold.

\begin{itemize}
\item[(i)] If $\alpha_{\theta}\pi_{\theta_{1}}v_{\theta_{1}}\not =0$, then for
the sequence $\pi_{\theta_{1}}\hat{Q}^{(\ell)}(\theta,n)\1 $, an
asymptotically equivalent sequence in the limit $n\rightarrow\infty$ is given
by
\begin{equation}
\alpha_{\theta}\pi_{\theta_{1}}v_{\theta_{1}}\rho(\theta)^{n+1-\kappa}%
\frac{n^{h^{+}(\theta)}}{h^{+}(\theta)!}\prod_{u\in H^{-}(\theta)}\frac
{1}{1-\frac{\rho_{\theta_{u}}}{\rho(\theta)}}\quad\mbox{ if }\rho_{\ell}%
=\rho(\theta) \label{N16}%
\end{equation}
and
\begin{equation}
\alpha_{\theta}\pi_{\theta_{1}}v_{\theta_{1}}\rho(\theta)^{n+1-\kappa}%
\frac{n^{h^{+}(\theta)-1}}{(h^{+}(\theta)-1)!}\frac{1}{1-\frac{\rho_{\ell}%
}{\rho(\theta)}}\prod_{u\in H^{-}(\theta)}\frac{1}{1-\frac{\rho_{\theta_{u}}%
}{\rho(\theta)}}\quad\mbox{ if }\rho_{\ell}<\rho(\theta)\,. \label{N17}%
\end{equation}

\item[(ii)] If $\alpha_{\theta}\pi_{\theta_{1}}v_{\theta_{1}}=0$, then the
exponential growth of the sequence $\pi_{\theta_{1}}\hat{Q}^{(\ell)}%
(\theta,n)\1 $ is for $\rho_{\ell}=\rho(\theta)$ equal to or less than
$\rho(\theta)^{n+1-\kappa}n^{h^{+}(\theta)-1}$ and for $\rho_{\ell}%
<\rho(\theta)$ equal to or less than $\rho(\theta)^{n+1-\kappa}n^{h^{+}%
(\theta)-2}$.
\end{itemize}
\end{proposition}

\begin{proof}
In a first step, we show that (\ref{4.19a}) holds for every $\theta
\in\mathcal{P}^{(\ell)}$, i.e.
\begin{align*}
&  \sum_{\hat{\eta}\in\Gamma_{\kappa+1}(n+1-\kappa)}Q_{\theta_{1}\theta_{1}%
}^{\eta_{1}}Q_{\theta_{1}\theta_{2}}\cdots Q_{\theta_{u-1}\ell}Q_{\ell\ell
}^{\eta_{\ell}}Q_{\ell\ell}^{\hat{\eta}_{\ell}}Q_{\ell\theta_{u+1}}\cdots
Q_{\theta_{\kappa-1}\theta_{\kappa}}Q_{\theta_{\kappa}\theta_{\kappa}}%
^{\eta_{\kappa}}\\
&  =\sum_{r=0}^{n}\sum_{\substack{\eta\in\Gamma_{\underline{\kappa}%
}(r+1-\underline{\kappa})\,,\\\zeta\in\Gamma_{\overline{\kappa}}%
(n-r+1-\overline{\kappa})}}Q_{\theta_{1}\theta_{1}}^{\eta_{1}}\cdots
Q_{\ell\ell}^{\eta_{\underline{\kappa}}}Q_{\ell\ell}^{\zeta_{1}}\cdots
Q_{\theta_{\kappa}\theta_{\kappa}}^{\zeta_{\overline{\kappa}}}\,.
\end{align*}
We order the summands on the left hand side by putting together the summands
with equal sum of the first $\underline{\kappa}$ exponents, say $\psi
_{1}+\dots+\psi_{\underline{\kappa}}=r+1-\underline{\kappa}$ for some
$r\in\{\underline{\kappa}-1,\dots,n+1-\overline{\kappa}\}$ where
$\overline{\kappa}=\kappa+1-\underline{\kappa}$. Hence, the sum of the last
$\overline{\kappa}$ exponents equals $n+1-\kappa-\left(  r+1-\underline
{\kappa}\right)  =n-r+1-\overline{\kappa}$. Then the left hand side equals%
\[
\sum_{r=\underline{\kappa}-1}^{n+1-\overline{\kappa}}\sum_{\substack{\eta
\in\Gamma_{\underline{\kappa}}(r+1-\underline{\kappa})\,,\\\zeta\in
\Gamma_{\overline{\kappa}}(n-r+1-\overline{\kappa})}}Q_{\theta_{1}\theta_{1}%
}^{\eta_{1}}\cdots Q_{\ell\ell}^{\eta_{\underline{\kappa}}}Q_{\ell\ell}%
^{\zeta_{1}}\cdots Q_{\theta_{\kappa}\theta_{\kappa}}^{\zeta_{\overline
{\kappa}}}\,.
\]
Since by definition $\Gamma_{\underline{\kappa}}(r+1-\underline{\kappa
})=\emptyset$ for $r+1-\underline{\kappa}<0$ and $\Gamma_{\overline{\kappa}%
}(n-r+1-\overline{\kappa})=\emptyset$ for $n-r+1-\underline{\kappa}<0$,
equality (\ref{4.19a}) follows.

We apply Proposition~\ref{prop1} to the matrix $\hat{Q}^{(\ell)}$ instead of
the matrix $Q$. For this extended matrix, we consider the admissible path of
length $\kappa+1$
\begin{equation}
\hat{\theta}^{(\ell)}:=(\theta_{1}+1,\dots,\theta_{u-1}+1,\ell+1,\ell
,\theta_{u+1},\dots,\theta_{\kappa}), \label{theta_hat}%
\end{equation}
where $\theta_{u}=\ell$. The matrix $\hat{Q}^{(\ell)}(\theta,n)$ corresponds
to the admissible sequence $\hat{\theta}^{(\ell)}$ for $\hat{Q}^{(\ell)}$,
more precisely, one sees that
\begin{equation}
\label{qhatformula}\hat{Q}^{(\ell)}(\theta,n)=(\hat{Q}^{(\ell)})(\hat{\theta
}^{(\ell)},n+1),
\end{equation}
since $\Gamma_{\kappa+1}(n+1-\kappa)=\Gamma_{\kappa+1}((n+1)+1-(\kappa+1))$,
where the right hand side of \eqref{qhatformula} is defined as
\eqref{P_Q_theta} with $Q$ replaced by $\hat{Q}^{(\ell)}$.

Proposition~\ref{prop1} yields the following two statements.

\begin{itemize}
\item[(i)] If $\pi_{\theta_{1}}v_{\theta_{1}}\not =0$ and $\alpha_{\theta
}\not =0$, then the sequence $\pi_{\theta_{1}}\hat{Q}^{(\ell)}(\theta
,n)\1 $ is asymptotically equivalent in the limit $n\rightarrow\infty$
to
\[
\pi_{\theta_{1}}v_{\theta_{1}}\alpha_{\theta}\sum_{\eta\in\Gamma_{\kappa
+1}(n+1-\kappa)}\rho_{\theta_{\kappa}}^{\eta_{\kappa}+1}\cdots\rho
_{\theta_{u+1}}^{\eta_{u+2}}\rho_{\ell}^{\eta_{u+1}}\rho_{\ell}^{\eta_{u}}%
\rho_{\theta_{u-1}}^{\eta_{u-1}}\cdots\rho_{\theta_{1}}^{\eta_{1}}\,,
\]
where $\alpha_{\theta}$ is equal to the corresponding quantity for the
matrices $Q$ and $\hat{Q}^{(\ell)}$, since for $\hat{Q}^{(\ell)}$ the block in
row $\ell+1$ and column $\ell$ is the identity matrix. If $\rho_{\ell}%
=\rho(\theta)$, then $h^{+}(\hat{\theta}^{(\ell)})=h^{+}(\theta)+1$, and if
$\rho_{\ell}< q(\theta)$, then $h^{+}(\hat{\theta}^{(\ell)})=h^{+}(\theta)$.
Hence, Proposition~\ref{prop1} shows that the sequence above is asymptotically
equivalent to
\[
\alpha_{\theta}\pi_{\theta_{1}}v_{\theta_{1}}\rho(\theta)^{n+1-\kappa}%
\frac{(n+1)^{h^{+}(\theta)}}{h^{+}(\theta)!}\prod_{u\in H^{-}(\theta)}\frac
{1}{1-\frac{\rho_{\theta_{u}}}{\rho(\theta)}}\quad\mbox{ if }\rho_{\ell}%
=\rho(\theta)
\]
and
\[
\alpha_{\theta}\pi_{\theta_{1}}v_{\theta_{1}}\rho(\theta)^{n+1-\kappa}%
\frac{(n+1)^{h^{+}(\theta)-1}}{(h^{+}(\theta)-1)!}\frac{1}{1-\frac{\rho_{\ell
}}{\rho(\theta)}}\prod_{u\in H^{-}(\theta)}\frac{1}{1-\frac{\rho_{\theta_{u}}%
}{\rho(\theta)}}\quad\mbox{ if }\rho_{\ell}<\rho(\theta)\,.
\]
This proves assertion (i), since $(n+1)^{h^{+}(\theta)}$ is asymptotically
equivalent to $n^{h^{+(\theta)}}$.

\item[(ii)] If $\pi_{\theta_{1}}v_{\theta_{1}}=0$ or $\alpha_{\theta}=0$, then
the exponential growth of the sequence $\pi_{\theta_{1}}\hat{Q}^{(\ell
)}(\theta,n)\1 $ is for $\rho_{\ell}=\rho(\theta)$ equal to or less
than $\rho(\theta)^{n+1-\kappa}n^{h^{+}(\theta)-1}$ and for $\rho_{\ell}%
<\rho(\theta)$ equal to or less than $\rho(\theta)^{n+1-\kappa}n^{h^{+}%
(\theta)-2}$.
\end{itemize}

This finishes the proof of the proposition.
\end{proof}

So far, we have fixed a particular admissible path $\theta\in\mathcal{P}$, and
in both Proposition~\ref{prop1} and Proposition~\ref{Nproposition4.24}, we
made certain assumptions relating to this particular $\theta$. In the
following assumption for our main results, we consider all relevant admissible
paths. Recall that $\rho_{\max} = \max\{\rho_{1},\dots,\rho_{k}\}$ and
$h^{+}_{\max}:=\max\{h^{+}(\theta): \theta\in\mathcal{P}\mbox{ and }
\rho(\theta)= \rho_{\max}\}$, and define the set of maximal admissible paths
$\mathcal{P}_{\max}$ by%
\[
\mathcal{P}_{\max}:=\left\{  \theta\in\mathcal{P}:h^{+}(\theta)=h^{+}_{\max}
\mbox{ and } \rho(\theta)=\rho_{\max} \right\}  \,,
\]
and let $\mathcal{P}_{\max}^{(\ell)}:=\mathcal{P}^{(\ell)}\cap\mathcal{P}%
_{\max}$.

\begin{assumption}
\label{mainass} Consider a matrix $Q$ of the Frobenius normal form
\eqref{P_Q}, and let $(X_{i})_{i\in\mathbb{N}_{0}}$ be the Markov chain
associated to the substochastic matrix $Q$ starting in $\pi$. We assume that

\begin{itemize}
\item[(i)] for all $i\in\{0,\dots,k\}$ with $\rho_{i}<\rho_{\max}$, the
diagonal term $Q_{ii}$ is scalar, and

\item[(ii)] there exists a maximal admissible path $\theta=(\theta_{1}%
,\dots,\theta_{\kappa})\in\mathcal{P}_{\max}$ such that the constant
$\alpha_{\theta}\not =0$ and the initial distribution $\pi$ satisfies
$\pi_{\theta_{1}}v_{\theta_{1}}\not =0$, where $\alpha_{\theta}$ and the
Perron--Frobenius eigenvector $v_{\theta_{1}}$ of $Q_{\theta_{1}\theta_{1}}$
are defined as in Lemma~\ref{lemma1}.
\end{itemize}
\end{assumption}

By combining Proposition~\ref{prop1} and Proposition~\ref{Nproposition4.24},
we arrive at the following formulas for the quasi-ergodic limits.

\begin{theorem}
[Quasi-ergodic limits for finite absorbing Markov chains]\label{Theorem3}%
Suppose that Assumption~\ref{mainass} holds, and let $\ell\in\{1,\dots,k\}$.
Then the following statements hold:

\begin{itemize}
\item[(i)] If $\rho_{\ell}<\rho_{\max}$ or $\mathcal{P}_{\max}^{(\ell
)}=\emptyset$, then%
\[
~\lim_{n\to\infty}\mathbb{E}_{\pi}\big[  \tfrac{1}{n+1}\#\big\{m\in
\{0,\dots,n\}:X_{m}\in I_{\ell}\big\}\big|T>n\big]  =0 \,.
\]

\item[(ii)] If $\rho_{\ell}=\rho_{\max}$ and $\mathcal{P}_{\max}^{(\ell
)}\not =\emptyset$, then%
\begin{align*}
&  ~\lim_{n\rightarrow\infty}\mathbb{E}_{\pi}\big[\tfrac{1}{n+1}%
\#\big\{m\in\{0,\dots,n\}:X_{m}\in I_{\ell}\big\}\big|T>n\big]\\
&  =\frac{\sum_{\theta\in\mathcal{P}_{\max}^{(\ell)}}\alpha_{\theta}%
\pi_{\theta_{1}}v_{\theta_{1}}\prod_{u\in H^{-}(\theta)}\frac{1}{\rho_{\max
}-\rho_{\theta_{u}}}}{h_{\max}^{+}\sum_{\theta\in\mathcal{P}_{\max}}%
\alpha_{\theta}\pi_{\theta_{1}}v_{\theta_{1}}\prod_{u\in H^{-}(\theta)}%
\frac{1}{\rho_{\max}-\rho_{\theta_{u}}}}\,.
\end{align*}

\end{itemize}
\end{theorem}

\begin{proof}
We first assume $\rho_{\ell}=\rho_{\max}$ and show that in this case, we have
\begin{align}
&  \lim_{n\rightarrow\infty}\mathbb{E}_{\pi}\big[\tfrac{1}{n+1}\#\big\{m\in
\{0,\dots,n\}:X_{m}\in I_{\ell}\big\}\big|T>n\big]\nonumber\\
=  &  \lim_{n\rightarrow\infty}\frac{\sum_{\theta\in\mathcal{P}^{(\ell)}}%
\pi_{\theta_{1}}\hat{Q}^{(\ell)}(\theta,n)\1 }{\sum_{\theta
\in\mathcal{P}}(n+1)\pi_{\theta_{1}}Q(\theta,n)\1 }\nonumber\\
=  &  \lim_{n\rightarrow\infty}\frac{\sum_{\theta\in\mathcal{P}_{\max}%
^{(\ell)}}\pi_{\theta_{1}}\hat{Q}^{(\ell)}(\theta,n)\1 }{\sum
_{\theta\in\mathcal{P}_{\max}}(n+1)\pi_{\theta_{1}}Q(\theta,n)\1 %
}\,.\label{thirdrow}%
\end{align}
The first equality follows from Corollary \ref{lemma2b}~(i) and Proposition
\ref{Nproposition4.24} applied to the numerator in Corollary \ref{lemma2b}~(i).

For the second equality, we identify summands in both denominator and
numerator that dominate for $n\rightarrow\infty$. By Assumption \ref{mainass}
(ii) there exists a maximal admissible path $\theta\in\mathcal{P}_{\max}$ with
$\alpha_{\theta}\pi_{\theta_{1}}v_{\theta_{1}}\not =0$, hence Proposition
\ref{prop1} shows that for the sequence $\pi_{\theta_{1}}Q(\theta
,n)\1 $ an asymptotically equivalent sequence for $n\rightarrow\infty$
is given by (\ref{asequ45}). Thus $(n+1)\pi_{\theta_{1}}Q(\theta,n)$ has
exponential growth rate equal to $\rho(\theta)^{n}n^{h^{+}(\theta)}=\rho
_{\max}^{n}n^{h_{\max}^{+}}$. For any $\theta\in\mathcal{P}\setminus
\mathcal{P}_{\max}$, one has $h^{+}(\theta)<h_{\max}^{+}$, and the summand
$(n+1)\pi_{\theta_{1}}Q(\theta,n)\1 $ grows at most with the smaller
exponential growth rate $\rho(\theta)^{n}n^{h^{+}(\theta)}$, again by
Proposition~\ref{prop1}. This justifies replacing $\mathcal{P}$ by
$\mathcal{P}_{\max}$ in the denominator. For any $\theta\in\mathcal{P}%
^{(\ell)}$ in the numerator, Proposition \ref{Nproposition4.24} shows
that the summand $\pi_{\theta_{1}}\hat{Q}^{(\ell)}(\theta
,n)\1 $ grows at most with exponential growth rate $\rho(\theta
)^{n}n^{h^{+}(\theta)}=\rho_{\max}^{n}n^{h^{+}(\theta)}$, see \eqref{N16} (the
exponential growth can be smaller than that by
Proposition~\ref{Nproposition4.24}~(ii), when $\alpha_{\theta}\pi_{\theta_{1}%
}v_{\theta_{1}}=0$). This justifies replacing $\mathcal{P}^{(\ell)}$ by
$\mathcal{P}_{\max}^{(\ell)}$ in the numerator using again $h^{+}%
(\theta)<h_{\max}^{+}$ for $\theta\in\mathcal{P}^{(\ell)}\setminus
\mathcal{P}_{\max}^{(\ell)}$ (note that for $\mathcal{P}_{\max}^{(\ell
)}=\varnothing$ the limits for $n\rightarrow\infty$ equal $0$).

Now for the denominator in \eqref{thirdrow}, Proposition~\ref{prop1} (i)
yields the following: Let $\theta\in\mathcal{P}_{\max}$ with $\alpha_{\theta
}\pi_{\theta_{1}}v_{\theta_{1}}\not =0$ (recall that the existence is clear
due to Assumption~\ref{mainass}~(ii)). Then for the corresponding summand,
given by $(n+1)\pi_{\theta_{1}}Q(\theta,n)\1 $, an asymptotically
equivalent sequence is
\[
\Psi_{n}(\theta)=\alpha_{\theta}\pi_{\theta_{1}}v_{\theta_{1}}\rho_{\max
}^{n+1-\kappa(\theta)}\frac{n^{h_{\max}^{+}}}{(h_{\max}^{+}-1)!}\prod_{u\in
H^{-}(\theta)}\frac{1}{1-\frac{\rho_{\theta_{u}}}{\rho_{\max}}}\,.
\]
Note that for $\theta\in\mathcal{P}_{\max}$ with $\alpha_{\theta}\pi
_{\theta_{1}}v_{\theta_{1}}=0$, the corresponding summand given by
$(n+1)\pi_{\theta_{1}}Q(\theta,n)\1 $ has weaker exponential growth for
$n\rightarrow\infty$, equal to or less than $\rho_{\max}^{n}n^{h_{\max}^{+}%
-1}$. Since $\Psi_{n}(\theta)=0$ for those $\theta$, this implies that an
asymptotically equivalent term for the denominator is given by $\sum
_{\theta\in\mathcal{P}_{\max}}\Psi_{n}(\theta)$ with exponential growth
$\rho_{\max}^{n}n^{h_{\max}^{+}}$.

For the numerator, suppose first that there exists $\theta\in\mathcal{P}%
_{\max}^{(\ell)}$ with $\alpha_{\theta}\pi_{\theta_{1}}v_{\theta_{1}}\not =0$.
Then by Proposition \ref{Nproposition4.24}, for the corresponding summand
$\pi_{\theta_{1}}\hat{Q}^{(\ell)}(\theta,n)\1 $, an asymptotically
equivalent sequence is
\[
\Xi_{n}(\theta):=\alpha_{\theta}\pi_{\theta_{1}}v_{\theta_{1}}\rho_{\max
}^{n+1-\kappa(\theta)}\frac{n^{h_{\max}^{+}}}{h_{\max}^{+}!}\prod_{u\in
H^{-}(\theta)}\frac{1}{1-\frac{\rho_{\theta_{u}}}{\rho_{\max}}}\,.
\]
By Proposition \ref{Nproposition4.24} (ii), for $\theta\in\mathcal{P}_{\max
}^{(\ell)}$ with $\alpha_{\theta}\pi_{\theta_{1}}v_{\theta_{1}}=0$ the
corresponding summand given by $\pi_{\theta_{1}}\hat{Q}^{(\ell)}%
(\theta,n)\1 $ has weaker exponential growth for $n\rightarrow\infty$,
equal to or less than $\rho_{\max}^{n}n^{h_{\max}^{+}-1}$. Since $\Xi
_{n}(\theta)=0$ for those $\theta$, this implies that an asymptotically
equivalent term for the numerator is given by $\sum_{\theta\in\mathcal{P}%
_{\max}^{(\ell)}}\Xi_{n}(\theta)$. Hence, the formula given in (ii) holds in
this case. Otherwise, $\mathcal{P}_{\max}^{(\ell)}=\emptyset$ or for all
$\theta\in\mathcal{P}_{\max}^{(\ell)}$ we have $\alpha_{\theta}\pi_{\theta
_{1}}v_{\theta_{1}}=0$. Then the limits for $n\rightarrow\infty$ are equal to
$0$, hence assertion (ii) and also the second statement in (i) follow.

It remains to show (i) under the assumption $\rho_{\ell}<\rho_{\max}$. In this
case, exactly like above, an asymptotically equivalent term for the
denominator is given by $\sum_{\theta\in\mathcal{P}_{\max}}\Psi_{n}(\theta)$.
Consider a summand in the numerator, so let $\theta\in\mathcal{P}^{(\ell)}$.
Then Proposition~\ref{Nproposition4.24} shows that the exponential growth for
$n\rightarrow\infty$ of the numerator is bounded above by either $\rho
(\theta)^{n}n^{h^{+}(\theta)}$ if $\rho(\theta)<\rho_{\max}$, or by
$\rho_{\max}^{n}n^{h_{\max}^{+}-1}$ if $\rho(\theta)=\rho_{\max}$. In both
cases, the exponential growth is weaker than for the denominator, determined
by $\sum_{\theta\in\mathcal{P}_{\max}}\Psi_{n}(\theta)$ with exponential
growth $\rho_{\max}^{n}n^{h_{\max}^{+}}$. This finishes the proof of the theorem.
\end{proof}

\begin{remark}
We note that due to Remark~\ref{genericity_alpha}, the condition $\alpha_{\theta}\not =0$ in
Assumption~\ref{mainass}~(ii) is generically satisfied, and  the condition $\pi_{\theta_{1}}v_{\theta_{1}}\not =0$ in this assumption
is not restrictive. In fact, suppose that for a given
initial distribution $\pi$, there is no $\theta\in\mathcal{P}_{\max}$
with $\pi_{\theta_{1}}v_{\theta_{1}}\not =0$. Then define%
\begin{align*}
h_{\max}^{+,\pi}  & :=\max\big\{h^{+}(\theta):\theta\in\mathcal{P}\text{ and }%
\rho(\theta)=\rho_{\max},\pi_{\theta_{1}}v_{\theta_{1}}\not =0\big\}\,,\\
\mathcal{P}_{\max}^{\pi}  & :=\big\{\theta\in\mathcal{P}:\rho(\theta)=\rho_{\max
}\text{ and }h^{+}(\theta)=h_{\max}^{+,\pi}\big\}\,.
\end{align*}
Then Theorem~\ref{Theorem3} remains valid with $\mathcal{P}_{\max}$ and
$\mathcal{P}^{(\ell)}$ replaced by $\mathcal{P}_{\max}^{\pi}$ and
$\mathcal{P}_{\max}^{(\ell),\pi}:=\mathcal{P}^{(\ell)}\cap\mathcal{P}_{\max
}^{\pi}$, respectively.
\end{remark}

The formula for the quasi-ergodic limit in the above theorem can be simplified
in certain special cases.

\begin{corollary}
\label{Corollary4.16} Assume that in the setting of Theorem~\ref{Theorem3},
there is only one maximal path, i.e.~$\mathcal{P}_{\max}=\{\theta=(\theta
_{1},\dots,\theta_{\kappa})\}$. If $\ell=\theta_{u}$ with $\rho_{\ell}%
=\rho_{\max}$ for some $u\in\{1,\dots,\kappa\}$, then
\[
~\lim_{n\rightarrow\infty}\mathbb{E}_{\pi}\left[  \tfrac{1}{n+1}%
\#\big\{m\in\{0,\dots,n\}:X_{m}\in I_{\ell}\big\}\big|T>n\right]  =\frac
{1}{h_{\max}^{+}}\,,
\]
and this limit vanishes whenever $\rho_{\ell}<\rho_{\max}$.
\end{corollary}

\begin{proof}
The assertion is an immediate consequence of Theorem~\ref{Theorem3}.
\end{proof}

In the scalar case, one obtains the following formulas which are directly
given in terms of the matrix $Q$.

\begin{corollary}
\label{Corollary6.37} Consider a matrix $Q$ of the Frobenius normal form
\eqref{P_Q}, and let $(X_{i})_{i\in\mathbb{N}_{0}}$ be the Markov chain
associated to the substochastic matrix $Q$ starting in $\pi$. We assume that
all submatrices $q_{ij}:=Q_{ij}$ are scalar and the initial distribution $\pi$
satisfies $\pi_{\theta_{1}}\not =0$ for some maximal admissible path
$\theta=(\theta_{1},\dots,\theta_{\kappa})\in\mathcal{P}_{\max}$. Then for all
$\ell\in\{1,\dots,k\}$, the following holds.

\begin{itemize}
\item[(i)] If $\rho_{\ell}<\rho_{\max}$ or $\mathcal{P}_{\max}^{(\ell
)}=\emptyset$, then%
\[
~\lim_{n\to\infty}\mathbb{E}_{\pi}\left[  \tfrac{1}{n+1}\#\big\{m\in
\{0,\dots,n\}:X_{m}\in I_{\ell}\big\}\big|T>n\right]  =0.
\]

\item[(ii)] If $\rho_{\ell}=\rho_{\max}$ and $\mathcal{P}_{\max}^{(\ell
)}\not =\emptyset$, then
\begin{align*}
&  ~\lim_{n\rightarrow\infty}\mathbb{E}_{\pi}\left[  \tfrac{1}{n+1}%
\#\big\{m\in\{0,\dots,n\}:X_{m}\in I_{\ell}\big\}\big|T>n\right] \\
&  =\frac{1}{h_{\max}^{+}}\frac{\sum_{\theta\in\mathcal{P}_{\max}^{(\ell)}}%
\pi_{\theta_{1}}q_{\theta_{1}\theta_{2}}\cdots q_{\theta_{\kappa-1}%
\theta_{\kappa}}\prod_{u\in H^{-}(\theta)}\frac{1}{\rho_{\max}-\rho
_{\theta_{u}}}}{\sum_{\theta\in\mathcal{P}_{\max}}\pi_{\theta_{1}}%
q_{\theta_{1}\theta_{2}}\cdots q_{\theta_{\kappa-1}\theta_{\kappa}}\prod_{u\in
H^{-}(\theta)}\frac{1}{\rho_{\max}-\rho_{\theta_{u}}}}\,.
\end{align*}

\end{itemize}
\end{corollary}

\begin{proof}
In the scalar case considered here, one finds, for an admissible path
$\theta=(\theta_{1},\dots,\theta_{\kappa})$, that the constants $\alpha
_{\theta_{u}}$ defined in Lemma~\ref{lemma1}~(iv) are given by $\alpha
_{\theta_{\kappa}}=1$ and $\alpha_{\theta_{u}}=q_{\theta_{u}\theta_{u+1}}$ for
$u\in\{\kappa-1,\kappa-2,\dots,1\}$. Hence, we get
\[
\alpha_{\theta}\pi_{\theta_{1}}v_{\theta_{1}}=\pi_{\theta_{1}}q_{\theta
_{1}\theta_{2}}\cdots q_{\theta_{\kappa-1}\theta_{\kappa}}\not =0\,,
\]
and thus, the assertion follows from Theorem~\ref{Theorem3}.
\end{proof}

\begin{remark}
\label{Remark_same_pi}In the scalar case, suppose that all maximal admissible
paths begin in the same element, i.e.~$\theta_{1}=\theta_{1}^{\prime}$ for all
$\theta,\theta^{\prime}\in\mathcal{P}_{\max}$. Then the quasi-ergodic limits
are independent of the initial distribution $\pi$. In fact, here all scalars
$\pi_{\theta_{1}}$ and $\pi_{\theta_{1}^{\prime}}$ coincide and Corollary
\ref{Corollary6.37} yields for $\rho_{\ell}=\rho_{\max}$ and $\mathcal{P}%
_{\max}^{(\ell)}\not =\emptyset$%
\begin{align*}
&  ~\lim_{n\rightarrow\infty}\mathbb{E}_{\pi}\big[\tfrac{1}{n+1}%
\#\big\{m\in\{0,\dots,n\}:X_{m}\in I_{\ell}\big\}\big|T>n\big]\\
&  =\frac{1}{h_{\max}^{+}}\frac{\sum_{\theta\in\mathcal{P}_{\max}^{(\ell)}%
}q_{\theta_{1}\theta_{2}}\cdots q_{\theta_{\kappa-1}\theta_{\kappa}}%
\prod_{u\in H^{-}(\theta)}\frac{1}{\rho_{\max}-\rho_{\theta_{u}}}}%
{\sum_{\theta\in\mathcal{P}_{\max}}q_{\theta_{1}\theta_{2}}\cdots
q_{\theta_{\kappa-1}\theta_{\kappa}}\prod_{u\in H^{-}(\theta)}\frac{1}%
{\rho_{\max}-\rho_{\theta_{u}}}}\,.
\end{align*}

\end{remark}

Next we discuss the behaviour within the blocks. In order to deal with the
numerator in Corollary \ref{lemma2b} (ii), we consider for $\ell\in
\{1,\dots,k\}$ instead of the matrix $\hat{Q}^{(\ell)}$ the following matrix
for any $t\in\{1,\dots,d_{\ell}\}$ and $e_{t}\in\mathbb{R}^{d_{\ell}}$,
\[
\hat{Q}^{(\ell,t)}:=%
\begin{pmatrix}
Q_{11} &  &  &  &  & 0\\
\vdots & \ddots &  &  &  & \\
Q_{\ell1} & \hdots & Q_{\ell\ell} &  &  & \\
&  & e_{t}e_{t}^{\top} & Q_{\ell\ell} &  & \\
&  &  & \vdots & \ddots & \\
0 &  &  & Q_{k\ell} & \hdots & Q_{kk}%
\end{pmatrix}
\,.
\]
This matrix is also of the form \eqref{P_Q} with one more block, i.e.~$k+1$
blocks. Denote the elements of $\Gamma_{\kappa+1}(n+1-\kappa)$ by $\hat{\eta
}=(\eta_{1},\dots,\eta_{\ell},\hat{\eta}_{\ell},\dots,\eta_{\kappa})$ and
define,
\[
\hat{Q}^{(\ell,t)}(\theta,n)=\hspace{-3ex}\sum_{\hat{\eta}\in\Gamma_{\kappa
+1}(n+1-\kappa)}Q_{\theta_{1}\theta_{1}}^{\eta_{1}}Q_{\theta_{1}\theta_{2}%
}\cdots Q_{\theta_{u-1}\ell}Q_{\ell\ell}^{\eta_{\ell}}e_{t}e_{t}^{\top}%
Q_{\ell\ell}^{\hat{\eta}_{\ell}}Q_{\ell\theta_{u+1}}\cdots Q_{\theta
_{\kappa-1}\theta_{\kappa}}Q_{\theta_{\kappa}\theta_{\kappa}}^{\eta_{\kappa}}.
\]
The following proposition applies to the terms in the numerator of Corollary
\ref{lemma2b}~(ii). It is convenient to change the notation slightly by
writing $u^{(i)}$ and $v^{(i)}$ instead of $u_{i}$ and $v_{i}$ for the left
and right eigenvectors, respectively, of the Perron--Frobenius eigenvalue of
$Q_{ii}$.

\begin{proposition}
\label{propblocks} Consider a matrix $Q$ of the form \eqref{P_Q}, numbers
$\ell\in\{1,\dots,k\}$ and $t\in\{1,\dots,d_{\ell}\}$, and an admissible path
$\theta=(\theta_{1},\dots,\theta_{\kappa})\in\mathcal{P}^{(\ell)}$.
Furthermore, suppose that $\rho_{\ell}=\rho(\theta)$ and that for all
$u\in\{1,\dots,\kappa\}$ with $\rho_{\theta_{u}}<\rho(\theta)$,
the diagonal term $Q_{\theta_{u}\theta_{u}}$ is scalar. Then
\begin{equation}
\hat{Q}^{(\ell,t)}(\theta,n)=\sum_{r=0}^{n}Q(\underline{\theta^{\ell}}%
,r)e_{t}e_{t}^{\top}Q(\overline{\theta^{\ell}},n-r),\label{N4t}%
\end{equation}
and the following two statements hold:

\begin{itemize}
\item[(i)] If $\alpha_{\theta}\pi_{\theta_{1}}v^{(\theta_{1})}\not =0$, then
for the sequence $\pi_{\theta_{1}}\hat{Q}^{(\ell,t)}(\theta,n)\1 $, an
asymptotically equivalent sequence in the limit $n\rightarrow\infty$ is given
by
\[
u_{t}^{(\ell)}v_{t}^{(\ell)}\alpha_{\theta}\pi_{\theta_{1}}v^{(\theta_{1}%
)}\rho(\theta)^{n-\kappa}\frac{n^{h^{+}(\theta)}}{h^{+}(\theta)!}%
\prod_{u\in H^{-}(\theta)}\frac{1}{1-\frac{\rho_{\theta_{u}}}{\rho(\theta)}%
}\,,
\]
where $u^{(\ell)\top}=(u_{1}^{(\ell)},\dots,u_{d_{\ell}}^{(\ell)})$ and
$v^{(\ell)}=(v_{1}^{(\ell)},\dots,v_{d_{\ell}}^{(\ell)})^{\top}$ are the
positive left and right eigenvector of $Q_{\ell\ell}$ for $\rho_{\ell}$,
respectively, normalised in the sense of $\sum_{i=1}^{d_{\ell}}v_{i}^{(\ell
)}=1$ and $u^{(\ell)\top}v^{(\ell)}=1$.

\item[(ii)] If $\alpha_{\theta}\pi_{\theta_{1}}v^{(\theta_{1})}=0$, then the
exponential growth of the sequence $\pi_{\theta_{1}}\hat{Q}^{(\ell,t)}%
(\theta,n)\1 $ is equal to or less than $\rho(\theta)^{n+1-\kappa
}n^{h^{+}(\theta)-1}$.
\end{itemize}
\end{proposition}

\begin{proof}
We proceed as in Proposition~\ref{Nproposition4.24}, where we have used
Proposition~\ref{prop1} with the matrix $\hat{Q}^{(\ell)}$ instead of the
matrix $Q$. Here we can argue analogously using the matrix $\hat{Q}^{(\ell
,t)}$ instead of $Q$.

In a first step we verify equality (\ref{N4t}), i.e.
\begin{align*}
&  \sum_{\hat{\eta}\in\Gamma_{\kappa+1}(n+1-\kappa)}Q_{\theta_{1}\theta_{1}%
}^{\eta_{1}}Q_{\theta_{1}\theta_{2}}\cdots Q_{\theta_{u-1}\ell}Q_{\ell\ell
}^{\eta_{\ell}}e_{t}e_{t}^{\top}Q_{\ell\ell}^{\hat{\eta}_{\ell}}Q_{\ell
\theta_{u+1}}\cdots Q_{\theta_{\kappa-1}\theta_{\kappa}}Q_{\theta_{\kappa
}\theta_{\kappa}}^{\eta_{\kappa}}\\
&  =\sum_{r=0}^{n}\sum_{\substack{\eta\in\Gamma_{\underline{\kappa}%
}(r+1-\underline{\kappa})\,,\\\zeta\in\Gamma_{\overline{\kappa}}%
(n-r+1-\overline{\kappa})}}Q_{\theta_{1}\theta_{1}}^{\eta_{1}}\cdots
Q_{\ell\ell}^{\eta_{\underline{\kappa}}}e_{t}e_{t}^{\top}Q_{\ell\ell}%
^{\zeta_{1}}\cdots Q_{\theta_{\kappa}\theta_{\kappa}}^{\zeta_{\overline
{\kappa}}}\,.
\end{align*}
This follows as in the proof of equality (\ref{4.19a}) in
Proposition~\ref{Nproposition4.24}.

Now, we will apply Proposition \ref{prop1} to the matrix $\hat{Q}^{(\ell
,t)}(\theta,n)$ using the admissible sequence $\hat{\theta}^{(\ell)}$ defined
in (\ref{theta_hat}). The only difference to the proof of
Proposition~\ref{Nproposition4.24} occurs when we determine the factor
$\alpha_{\theta}$ from Proposition~\ref{prop1}. In the earlier proof, the
block in row $\ell+1$ and column $\ell$ of the matrix $\hat{Q}^{(\ell)}$ is
the identity matrix, hence this did not lead to a change in this factor in
Proposition~\ref{Nproposition4.24}. Here, however, the identity matrix is
replaced by $e_{t}e_{t}^{\top}$, and we will show that we get the additional
factor $\tilde{\alpha}=u_{t}^{(\ell)}v_{t}^{(\ell)}$.

To determine this additional factor $\tilde{\alpha}$, we have to uniquely
decompose
\[
e_{t}e_{t}^{\top}v^{(\ell)}=\tilde{\alpha}v^{(\ell)}+\tilde{w}\,,
\]
where $\tilde{\alpha}\in\mathbb{R}$ and $\tilde{w}\in V_{\ell}^{-}$ as in
Lemma~\ref{lemma1}~(iv). We show now that with $\tilde{\alpha}=u_{t}^{(\ell
)}v_{t}^{(\ell)}$, we have
\[
\tilde{w}=e_{t}e_{t}^{\top}v^{(\ell)}-u_{t}^{(\ell)}v_{t}^{(\ell)}v^{(\ell
)}\in V_{\ell}^{-}\,.
\]
To see this, we determine the exponential growth of $Q_{\ell\ell}^{n}\tilde
{w}$ in the limit $n\rightarrow\infty$. Using that $e_{t}e_{t}^{\top}%
v^{(\ell)}=(0,\dots,0,v_{t}^{(\ell)},0,\dots,0)^{\top}$, we have
\[
Q_{\ell\ell}^{n}\tilde{w}=Q_{\ell\ell}^{n}(0,\dots,0,v_{t}^{(\ell)}%
,0,\dots,0)^{\top}-u_{t}^{(\ell)}v_{t}^{(\ell)}\underbrace{Q_{\ell\ell}%
^{n}v^{(\ell)}}_{=\rho_{\ell}^{n}v^{(\ell)}}\,.
\]
We multiply with the left eigenvector $u^{(\ell)\top}$ and get
\begin{align*}
u^{(\ell)\top}Q_{\ell\ell}^{n}\tilde{w}  &  =u^{(\ell)\top}Q_{\ell\ell}%
^{n}(0,\dots,0,v_{t}^{(\ell)},0,\dots,0)^{\top}-u_{t}^{(\ell)}v_{t}^{(\ell
)}\rho_{\ell}^{n}u^{(\ell)\top}v^{(\ell)}\\
&  =\rho_{\ell}^{n}u^{(\ell)\top}(0,\dots,0,v_{t}^{(\ell)},0,\dots,0)^{\top
}-u_{t}^{(\ell)}v_{t}^{(\ell)}\rho_{\ell}^{n}=0\,.
\end{align*}
This shows that $Q_{\ell\ell}^{n}\tilde{w}$ is in the orthogonal complement of
the one-dimensional space $\operatorname{span}(u^{(\ell)})$. Since
$\operatorname{span}(v^{(\ell)})$ is not orthogonal to $\operatorname{span}%
(u^{(\ell)})$, it follows that $\tilde{w}\in V_{\ell}^{-}$, because otherwise,
the component relating to $\operatorname{span}(v^{(\ell)})$ would become dominant.
\end{proof}

This observation can be used to determine the quasi-ergodic behaviour within
the blocks.

\begin{theorem}
\label{Theorem4}Assume the setting of Theorem~\ref{Theorem3}, and let $\ell
\in\{1,\dots,k\}$ such that $\rho_{\ell}=\rho_{\max}$ and $\mathcal{P}_{\max
}^{(\ell)}\not =\emptyset$. Then for all $s\in I_{\ell}$, we have
\begin{align*}
&  ~\lim_{n\rightarrow\infty}\mathbb{E}_{\pi}\left[  \tfrac{1}{n+1}%
\#\big\{m\in\{0,\dots,n\}:X_{m}=s\big\}\big|T>n\right] \\
&  =u_{t(s)}^{(\ell)}v_{t(s)}^{(\ell)}\frac{1}{h_{\max}^{+}}\frac{\sum
_{\theta\in\mathcal{P}_{\max}^{(\ell)}}\alpha_{\theta}\pi_{\theta_{1}%
}v_{\theta_{1}}\prod_{u\in H^{-}(\theta)}\frac{1}{\rho_{\max}-\rho_{\theta
_{u}}}}{\sum_{\theta\in\mathcal{P}_{\max}}\alpha_{\theta}\pi_{\theta_{1}%
}v_{\theta_{1}} \prod_{u\in H^{-}(\theta)}\frac{1}{\rho_{\max}-\rho
_{\theta_{u}}}}\,,
\end{align*}
where $t(s):=s-\sum_{i=1}^{\ell-1}d_{i}$. Here $u^{(\ell)\top}=(u_{1}^{(\ell
)},\dots,u_{d_{\ell}}^{(\ell)})$ and $v^{(\ell)}=(v_{1}^{(\ell)}%
,\dots,v_{d_{\ell}}^{(\ell)})^{\top}$ are the positive left and right
eigenvector of $Q_{\ell\ell}$ for $\rho_{\ell}$, respectively, normalised in
the sense of $\sum_{i=1}^{d_{\ell}}v_{i}^{(\ell)}=1$ and $u^{(\ell)\top
}v^{(\ell)}=1$.
\end{theorem}

\begin{proof}
The proof is the same as the proof of Theorem~\ref{Theorem3}, with the
difference that instead of Proposition~\ref{Nproposition4.24}, we need to use
Proposition~\ref{propblocks} here.
\end{proof}

Finally, we consider the case that some of the blocks $Q_{ii}$ are not
eventually positive.

\begin{remark}
[The irreducible case]\label{Remark_periodic}Consider a matrix $Q$ in
Frobenius normal form \eqref{P_Q} with diagonal matrices $Q_{ii}$ which are
irreducible and possibly periodic. Then there exists an $N\in\mathbb{N}$ such
that the diagonal blocks of $Q^{N}$ are eventually positive. Denote by
$\tilde{X}$ the Markov chain induced by $Q^{N}$ and the corresponding stopping
time by $\tilde{T}$. We get for any $j\in\{1,\dots,k\}$,%
\begin{align*}
&  \lim_{n\rightarrow\infty}\mathbb{E}_{\pi}\left[  \tfrac{1}{n+1}%
\#\big\{m\in\{0,\dots,n\}:X_{m}=j\big\}\big|T\geq n\big\}\right] \\
&  =\lim_{n\rightarrow\infty}\sum_{i=0}^{N-1}\mathbb{E}_{\pi}\left[  \tfrac
{1}{n+1}\#\big\{m\in N\mathbb{Z}+i:m\leq n\mbox{ and }X_{m}=j\big\}\big|T\geq
n\big\}\right] \\
&  =\lim_{n\rightarrow\infty}\sum_{i=0}^{N-1}\mathbb{E}_{\pi Q^{i}}\left[
\tfrac{1}{n+1}\#\big\{m\in N\mathbb{Z}+i:m\leq n\mbox{ and }\tilde{X}%
_{\frac{m-i}{N}}=j\big\}\big|T\geq n\big\}\right] \\
&  =\lim_{n\rightarrow\infty}\sum_{i=0}^{N-1}\mathbb{E}_{\pi Q^{i}}\left[
\tfrac{1}{Nn+1}\#\big\{m\in\{0,\dots,n\}:\tilde{X}_{m}=j\big\}\big|T\geq
Nn\big\}\right] \\
&  =\frac{1}{N}\sum_{i=0}^{N-1}\lim_{n\rightarrow\infty}\mathbb{E}_{\pi Q^{i}%
}\left[  \tfrac{1}{n+1}\#\big\{m\in\{0,\dots,n\}:\tilde{X}_{m}%
=j\big\}\big|\tilde{T}\geq n\big\}\right]  \,.
\end{align*}
Hence corresponding formulas for the case of not eventually positive matrices
follow from the formulas for the eventually positive case.
\end{remark}

\section{Examples}\label{sec4}

We present a number of examples illustrating
the previous results and we start with the following simple example from Bena\"{\i}m, Cloez, Panloup
\cite[Example 3.5]{BenaCP16}, slightly modified in order to get a lower
diagonal matrix $Q$.

\begin{example}
\label{Example_BCP}Consider the transition matrix given by%
\[
Q=\begin{pmatrix}
\rho_{1} & 0\\
1-\rho_{2} & \rho_{2}%
\end{pmatrix}\,,
\quad \text{ where }0<\rho_{1}<\rho_{2}<1\,.
\]
This example is reducible with scalar blocks and $k=2$. The maximal admissible
paths are $\theta=(2)$ with $\kappa(\theta)=1$ and $\theta^{\prime}=(2,1)$
with $\kappa(\theta^{\prime})=2$, hence $\mathcal{P}_{\max}%
=\{(2),(2,1)\}=\mathcal{P}_{\max}^{(2)}$ and $\mathcal{P}_{\max}^{(1)}=(2,1)$.
Furthermore, $H^{+}(\theta)=\{2\}=H^{+}(\theta^{\prime})$ and $H^{-}%
(\theta^{\prime})=\{1\}$ and hence $h_{\max}^{+}=1$. Applying Corollary~\ref{Corollary6.37} to this example, one finds that for $\pi_{2}>0$, the
quasi-ergodic limit for $\ell=1$ is given by
\[
~\lim_{n\rightarrow\infty}\mathbb{E}_{\pi}\left[  \tfrac{1}{n+1}%
\#\big\{m\in\{0,\dots,n\}:X_{m}=1\big\}\big|T>n\right]  =0\,,
\]
since $\rho_1<\rho_{\max}$, and for $\ell=2$, we have
\begin{displaymath}
 ~\lim_{n\rightarrow\infty}\mathbb{E}_{\pi}\left[  \tfrac{1}{n+1}%
\#\big\{m\in\{0,\dots,n\}:X_{m}=2\big\}\big|T>n\right] = 1\,,
\end{displaymath}
since these two probabilities sum up to 1.
The classification of the quasi-stationary distributions in van Doorn and
Pollett \cite{vanDooPol09}, in particular \cite[Theorem 4.3]{vanDooPol09},
shows that for this example there is a unique quasi-stationary distribution
given by the normalised left eigenvector $u^{\top}=\big(  \frac{1-\rho_{2}%
}{1-\rho_{1}},\frac{\rho_{2}-\rho_{1}}{1-\rho_{1}}\big)$ to the eigenvalue $\rho_2$ of the matrix~$Q$.
\end{example}

Next we consider a class of three-dimensional matrices.

\begin{example}\label{exam42}
Let $\rho\in(0,1)$, and consider
\[
Q=\begin{pmatrix}
\rho & 0 & 0\\
q_{21} & \rho & 0\\
q_{31} & q_{32} & \rho
\end{pmatrix}\,.
\]
In this scalar case, one finds that under the condition $q_{21},q_{32}\not =0$, one has
\[
\mathcal{P}_{\max}=\{\theta=(3,2,1)\}\quad\text{ with }\quad h_{\max}^{+}=h^{+}%
(\theta)=3\,.
\]
Corollary \ref{Corollary4.16} yields that for $\ell\in\{1,2,3\}$, the quasi-ergodic
limits are independent of the initial distribution $\pi$, and we have
\[
\lim_{n\rightarrow\infty}\mathbb{E}_{\pi}\left[  \frac{1}{n+1}\#\{m\in
\{0,\dots,n\}:X_{m}=\ell\}\left\vert T\geq n\right.  \}\right]  =\frac{1}{h_{\max
}^{+}}=\frac{1}{3}.
\]
The situation is different if instead we assume that $q_{21},q_{31}\not =0$ and $q_{32}=0$. Then
\begin{align*}
\mathcal{P}_{\max}  &  =\{\theta=(3,1),\theta^{\prime}=(2,1)\}=\mathcal{P}%
_{\max}^{(1)}\,,\\
h_{\max}^{+}  &  =h^{+}(\theta)=h^{+}(\theta^{\prime})=2\text{ and }%
\kappa(\theta)=\kappa(\theta^{\prime})=2\,.
\end{align*}
Here\ $\mathcal{P}_{\max}^{(2)}=\{\theta^{\prime}\}$ and $\mathcal{P}_{\max
}^{(3)}=\{\theta\}$. Corollary \ref{Corollary6.37} yields that for $\pi
_{1}<1$,
\begin{align*}
~\lim_{n\rightarrow\infty}\mathbb{E}_{\pi}\left[  \tfrac{1}{n+1}%
\#\big\{m\in\{0,\dots,n\}:X_{m}=1\big\}\big|T>n\right]   &  =\frac{1}{2}\,,\\
~\lim_{n\rightarrow\infty}\mathbb{E}_{\pi}\left[  \tfrac{1}{n+1}%
\#\big\{m\in\{0,\dots,n\}:X_{m}=2\big\}\big|T>n\right]   &  =\frac{1}{2}%
\frac{\pi_{2}q_{21}}{\pi_{2}q_{21}+\pi_{3}q_{31}}\,,\\
~\lim_{n\rightarrow\infty}\mathbb{E}_{\pi}\left[  \tfrac{1}{n+1}%
\#\big\{m\in\{0,\dots,n\}:X_{m}=3\big\}\big|T>n\right]   &  =\frac{1}{2}%
\frac{\pi_{3}q_{31}}{\pi_{2}q_{21}+\pi_{3}q_{31}}\,.
\end{align*}
Thus, the quasi-ergodic measure is%
\[
\left(  \frac{1}{2},\frac{1}{2}\frac{\pi_{2}q_{21}}{\pi_{2}q_{21}+\pi
_{3}q_{31}},\frac{1}{2}\frac{\pi_{3}q_{31}}{\pi_{2}q_{21}+\pi_{3}q_{31}%
}\right)  .
\]
This case illustrates in particular that the lower diagonal entries $Q_{ij}$ for $i>j$ are relevant for the quasi-ergodic limits, which also may depend on the
initial distribution $\pi$.
\end{example}

In the next example, a Perron--Frobenius eigenvalue $\rho_{i}<\rho_{\max}$ is
present, and the two maximal admissible paths start in the same element and have
different lengths.

\begin{example}
\label{Example8.3_NEW}Consider%
\[
Q=\begin{pmatrix}
\rho & 0 & 0 & 0\\
0 & \rho & 0 & 0\\
0 & q_{32} & \rho_{3} & 0\\
q_{41} & 0 & q_{43} & \rho
\end{pmatrix}
\]
with $0<\rho_{3}<\rho<1$ and $q_{32},q_{41},q_{43}\not =0$. We have
\begin{align*}
\mathcal{P}_{\max}  &  =\{\theta=(4,3,2),\theta^{\prime}=(4,1)\}=\mathcal{P}%
_{\max}^{(4)},\\
h_{\max}^{+}  &  =h^{+}(\theta)=h^{+}(\theta^{\prime})=2\text{ and }%
\kappa(\theta)=3,\kappa(\theta^{\prime})=2,
\end{align*}
and%
\[
\mathcal{P}_{\max}^{(1)}=\{\theta^{\prime}=(4,1)\},\mathcal{P}_{\max}%
^{(2)}=\mathcal{P}_{\max}^{(3)}=\{\theta=(4,3,2)\}.
\]
Since both maximal paths start in the same element, Remark
\ref{Remark_same_pi} yields that the quasi-ergodic limits do not depend on
$\pi$ provided $\pi_{4}\not =0$. Corollary \ref{Corollary6.37} yields that in this case, the quasi-ergodic measure is given by
\[
\left(  \frac{1}{2}\frac{q_{41}}{q_{41}+q_{43}q_{32}\frac{1}{\rho-\rho_{3}}},\frac{1}{2}\frac{q_{43}q_{32}\frac{1}{\rho-\rho_{3}}}{q_{41}%
+q_{43}q_{32}\frac{1}{\rho-\rho_{3}}},0,\frac{1}{2}\right)  .
\]
\end{example}

In the next example, three maximal paths are present, they have different
lengths, and start in different elements.

\begin{example}
Consider%
\[
Q=%
\begin{pmatrix}
\rho_{1} & 0 & 0 & 0 & 0\\
0 & \rho & 0 & 0 & 0\\
q_{31} & 0 & \rho & 0 & 0\\
0 & 0 & q_{43} & \rho & 0\\
0 & q_{52} & 0 & 0 & \rho
\end{pmatrix}
\]
with $0<\rho_{1}<\rho<1$ and $q_{31},q_{43},q_{52}\not =0$. We have
\[
\mathcal{P}_{\max}=\{(5,2),(4,3),(4,3,1)\},h_{\max}^{+}=2,\kappa
(5,2)=\kappa(4,3)=2,\kappa(4,3,1)=3,
\]
and%
\[
\mathcal{P}_{\max}^{(1)}=\{(4,3,1)\},\mathcal{P}_{\max}^{(2)}=\mathcal{P}%
_{\max}^{(5)}=\{(5,2)\},\mathcal{P}_{\max}^{(3)}=\mathcal{P}_{\max}%
^{(4)}=\{(4,3),(4,3,1)\}.
\]
Suppose that $\pi_{4}\not=0$ or $\pi_{5}\not =0$. Corollary \ref{Corollary6.37} yields
that%
\[
~\lim_{n\rightarrow\infty}\mathbb{E}_{\pi}\left[  \tfrac{1}{n+1}%
\#\big\{m\in\{0,\dots,n\}:X_{m}=1\big\}\big|T\geq n\big\}\right]  =0,
\]
and for $\ell=2,5$, we get
\begin{align*}
& \lim_{n\rightarrow\infty}\mathbb{E}_{\pi}\left[  \tfrac{1}{n+1}%
\#\big\{m\in\{0,\dots,n\}:X_{m}=\ell\big\}\big|T\geq n\big\}\right] \\
& = \frac
{1}{2}\frac{\pi_{5}q_{52}}{\pi_{4}q_{43}+\pi_{4}q_{43}q_{31}\frac
{1}{\rho-\rho_{1}}+\pi_{5}q_{52}}\,,
\end{align*}
and for $\ell=3,4$, we get
\begin{align*}
&\lim_{n\rightarrow\infty}\mathbb{E}_{\pi}\left[  \tfrac{1}{n+1}%
\#\big\{m\in\{0,\dots,n\}:X_{m}=3\big\}\big|T\geq n\big\}\right] \\
& =\frac{1}%
{2}\frac{\pi_{4}q_{43}+\pi_{4}q_{43}q_{31}\frac{1}{\rho-\rho_{1}}}{\pi_{4}q_{43}+\pi_{4}q_{43}q_{31}\frac{1}{\rho-\rho_{1}}+\pi_{5}q_{52}}\,.
\end{align*}
Suppose, for instance, that the diagonal terms are given by $\rho_{1}=0.5,\rho=0.75$
and $q_{31}=q_{43}=q_{52}=0.1$, and the initial distribution is of the form
$\pi=(\ast,\ast,\ast,0.5,0.3)$. Then an evaluation of the formulas above
yields the quasi-ergodic measure
\[
(0,0.15,0.35,0.35,0.15)\,.
\]
\end{example}

Finally, we present an example with a non-scalar diagonal matrix $Q_{\ell\ell
}\in\mathbb R^{d_\ell\times d_\ell}$, where $d_{\ell}>1$.

\begin{example}
Consider the matrix
\[
Q= \begin{pmatrix}
Q_{11} & 0\\
Q_{21} & \rho_{2}%
\end{pmatrix}
\in\mathbb{R}^{3\times3}\,,
\]
with an eventually positive matrix $Q_{11}\in\mathbb{R}^{2\times2}%
,0\not =Q_{21}\in\mathbb{R}^{1\times2}$, and $\rho_{2}>0$. Let the eigenvalues
$\rho_{1},\rho_{1}^{-}$ of $Q_{11}$ satisfy $\rho:=\rho_{1}>\left\vert
\rho_{1}^{-}\right\vert $ and $\rho_{1}>\rho_{2}$. Here the index sets are
$I_{1}=\{1,2\}$ and $I_{2}=\{3\}$, and we have
\[
\mathcal{P}_{\max}=\mathcal{P}_{\max}^{(1)}=\{(1),(2,1)\},\mathcal{P}_{\max
}^{(2)}=\{(2,1)\}\text{ and }h_{\max}^{+}=1\,.
\]
Suppose that $\alpha_{\theta}\not =0$ for $\theta=(1)$ or $(2,1)$ and the
initial distribution $\pi=(\pi^{(1)},\pi^{(2)})$ satisfies $\pi^{(2)}\not =0$.
Then Theorem \ref{Theorem4} yields
\[
\lim_{n\rightarrow\infty}\mathbb{E}_{\pi}\left[  \tfrac{1}{n+1}%
\#\big\{m\in\{0,\dots,n\}:X_{m}=3\big\}\big|T>n\right]  =0\,.
\]
Furthermore, let $u^{(1)}$ and $v^{(1)}$ be the left and right eigenvectors,
respectively, for the eigenvalue $\rho$ of $Q_{11}$, normalised by $v_{1}^{(1)}%
+v_{2}^{(1)}=1$ and $u^{(1)\top}v^{(1)}=1$. Using $\mathcal{P}_{\max
}=\mathcal{P}_{\max}^{(1)}$ one obtains for $s\in\{1,2\}$ (with $t(s)=s$) that
\begin{equation}
~\lim_{n\rightarrow\infty}\mathbb{E}_{\pi}\left[  \tfrac{1}{n+1}%
\#\big\{m\in\{0,\dots,n\}:X_{m}=s\big\}\big|T>n\right]  =u_{s}^{(1)}%
v_{s}^{(1)}. \label{Ex6.5a}%
\end{equation}
As a specific example, consider the eventually positive matrix%
\[
Q_{11}=\begin{pmatrix}
0.2 & 0.1\\
0.1 & 0
\end{pmatrix}
\quad \text{ with }\left(  Q_{11}\right)  ^{2}= \begin{pmatrix}
0.05 & 0.02\\
0.02 & 0.01
\end{pmatrix}
\]
and eigenvalues $\rho=\rho_{1}=0.1(1+\sqrt{2})$ and $\rho_{1}^{-}%
=0.1(1-\sqrt{2})$. The left and right normalised eigenvectors of $\rho$ are
\[
u^{(1)\top}=\frac{1}{2}\big(1+\sqrt{2},1\big)\quad\text{and}\quad v^{(1)}=\frac{1}{2+\sqrt{2}%
}\big(1+\sqrt{2},1\big)^{\top}.
\]
For $\theta=(1)$, one finds that the constant $\alpha_{\theta}=\alpha_{(1)}\not =0$, it is determined by the decomposition%
\[
\1 =\begin{pmatrix}
1\\
1
\end{pmatrix}
=\alpha_{\theta}v^{(1)}+w_{\theta}=\alpha_{(1)}\frac{1}{2+\sqrt{2}%
}\begin{pmatrix}
1+\sqrt{2}\\
1
\end{pmatrix}
+w_{\theta}\,,
\]
with $w_{\theta}$ in the eigenspace for the eigenvalue $\rho^{-}$ of $Q_{11}$.
Thus, for an initial distribution with $\pi^{(2)}\not =0$, one obtains from
\eqref{Ex6.5a} that
\begin{align*}
~\lim_{n\rightarrow\infty}\mathbb{E}_{\pi}\left[  \tfrac{1}{n+1}%
\#\big\{m\in\{0,\dots,n\}:X_{m}=1\big\}\big|T>n\right]   &  =\frac{3+2\sqrt
{2}}{4+2\sqrt{2}}\,,\\
~\lim_{n\rightarrow\infty}\mathbb{E}_{\pi}\left[  \tfrac{1}{n+1}%
\#\big\{m\in\{0,\dots,n\}:X_{m}=2\big\}\big|T>n\right]   &  =\frac{1}%
{4+2\sqrt{2}}\,.
\end{align*}

\end{example}

\appendix
\section{Proof of Proposition \ref{P_prop1}}

The proof of Proposition~\ref{P_prop1} provided in this appendix is based on methods from Darroch and Seneta~\cite{DarrS65}.

The following preparations are necessary.

\begin{lemma}
For any $j\in\{1,\dots,d\}$ and $n\in\mathbb{N}_{0}$, we have
\[
\sum_{\ell=0}^{n}\ell\,\mathbb{P}_{\pi}\big[\#\big\{m\in\{0,\dots
,n\}:X_{m}=j\big\}=\ell\mbox{ and }T=n+1\big]=\frac{\mathrm{d}}{\mathrm{d}%
z}\pi_{j}(z)Q_{j}^{n}(z)R\Big|_{z=1}\,.
\]

\end{lemma}

\begin{proof}
For a given finite sequence $(x_{0},x_{1},\dots,x_{n})\in\{1,\dots,d\}^{n+1}$,
the probability $\mathbb{P}_{\pi}[X_{i}=x_{i}\mbox{ for all }i\in
\{0,\dots,n\}\mbox{ and }T=n\big]$ is given by
\[
\pi_{x_{0}}q_{x_{0},x_{1}}q_{x_{1},x_{2}}\dots q_{x_{n-1},x_{n}}r_{x_{n}}\,,
\]
where $q_{ij}$ denote the entries of the matrix $Q$, and
\[
\pi Q^{n}R=\sum_{x_{0}=1}^{d}\sum_{x_{1}=1}^{d}\dots\sum_{x_{n}=1}^{d}%
\pi_{x_{0}}q_{x_{0},x_{1}}q_{x_{1},x_{2}}\dots q_{x_{n-1},x_{n}}r_{x_{n}}\,.
\]
Define for $j\in\{0,\dots,d\}$
\[
\gamma_{x_{0},x_{1},\dots,x_{n}}(z):=z^{\#\{m\in\{0,\dots,n\}:x_{m}=j\}}%
\pi_{x_{0}}q_{x_{0},x_{1}}q_{x_{1},x_{2}}\dots q_{x_{n-1},x_{n}}r_{x_{n}}\,.
\]
Then it follows that
\[
\pi_{j}(z)Q_{j}^{n}(z)R=\sum_{x_{0}=1}^{d}\sum_{x_{1}=1}^{d}\dots\sum
_{x_{n}=1}^{d}\gamma_{x_{0},x_{1},\dots,x_{n}}(z)\,.
\]
Hence, $z\mapsto\pi_{j}(z)Q_{j}^{n}(z)R$ is the probability generating
function for $\ell\in\{0,1,\dots,n\}$ having the probability $\mathbb{P}_{\pi
}\big[\#\big\{m\in\{0,\dots,n\}:X_{m}=j\big\}=\ell\mbox{ and }T=n+1\big]$.
This means that its expectation is given by $\frac{\mathrm{d}}{\mathrm{d}z}%
\pi_{j}(z)^{\top}Q_{j}^{n}(z)R\big|_{z=1}$, which finishes the proof of this lemma.
\end{proof}

We actually do not need this lemma, but the following lemma, which can be
proved analogously.

\begin{lemma}
\label{P_lemma1}For any $j\in\{1,\dots,d\}$ and $n\in\mathbb{N}_{0}$, we have
\[
\sum_{\ell=0}^{n}\ell\,\mathbb{P}_{\pi}\big[\#\big\{m\in\{0,\dots
,n\}:X_{m}=j\big\}=\ell\mbox{ and }T>n\big]=\frac{\mathrm{d}}{\mathrm{d}z}%
\pi_{j}(z)Q_{j}^{n}(z)\1 \Big|_{z=1}\,.
\]

\end{lemma}

We now start the proof of Proposition \ref{P_prop1}.

\begin{proof}
[Proof of Proposition \ref{P_prop1}](i) Note first that $\mathbb{P}_{\pi}(T>
n)=\pi Q^{n}\1 $. We get%

\begin{align*}
&  \mathbb{E}_{\pi}\big[\tfrac{1}{n+1}\#\big\{m\in\{0,\dots,n\}:X_{m}%
=j\big\}\big|T>n\big]\\
&  =\tfrac{1}{n+1}\sum_{\ell=0}^{n}\ell\,\frac{\mathbb{P}_{\pi}%
\big[\#\big\{m\in\{0,\dots,n\}:X_{m}=j\big\}=\ell\mbox{ and }T>n\big]}%
{\mathbb{P}_{\pi}(T>n)}\\
&  \overset{\text{Lemma }\ref{P_lemma1}}{=}\frac{\frac{\mathrm{d}}%
{\mathrm{d}z}\pi_{j}(z)Q_{j}^{n}(z)\1 \Big|_{z=1}}{(n+1)\pi
Q^{n}\1 }.
\end{align*}

(ii) The product rule implies%
\[
\frac{\mathrm{d}}{\mathrm{d}z}Q_{j}^{n}(z)=\sum_{m=0}^{n-1}Q_{j}^{m}(z)\left(
\frac{\mathrm{d}}{\mathrm{d}z}Q_{j}(z)\right)  Q_{j}^{n-1-m}(z),
\]
and hence
\[
\frac{\mathrm{d}}{\mathrm{d}z}Q_{j}^{n}(z)\Big|_{z=1}=\sum_{m=0}^{n-1}%
Q^{m}q_{j}e_{j}^{\top}Q^{n-1-m},
\]
where $q_{j}$ is the $j$th column of $Q$, and $e_{j}$ is the $j$th unit
vector. If $Q$ is eventually positive, Seneta \cite[Theorem 1.2]{Sene06}
implies that
\begin{equation}
Q^{m}=\rho^{m}vu^{\top}+O(m^{d-1}|\rho^{\prime}|^{m})\,. \label{P_eqn1}%
\end{equation}
This yields
\begin{align}
\frac{\mathrm{d}}{\mathrm{d}z}Q_{j}^{n}(z)\Big|_{z=1}  &  =\sum_{m=0}%
^{n-1}\Big(\rho^{n-1}vu^{\top}q_{j}e_{j}^{\top}vu^{\top}+O(\rho^{m}%
|\rho^{\prime}|^{n-1-m}(n-1-m)^{d-1})\nonumber\\
&  ~+O(\rho^{n-1-m}|\rho^{\prime}|^{m}m^{d-1})+O(m^{d-1}|\rho^{\prime}%
|^{m})O((n-1-m)^{d}|\rho^{\prime}|^{n-1-m}\Big)\label{P_eqn2}\\
&  =n\rho^{n}u_{j}v_{j}vu^{\top}+O(n^{2}\rho^{n}),\nonumber
\end{align}
since $u^{\top}q_{j}=\rho u_{j}$ and $e_{j}^{\top}v=v_{j}$. Thus,
\begin{align*}
&  \frac{\frac{\mathrm{d}}{\mathrm{d}z}\pi_{j}(z)Q_{j}^{n}(z)\1 %
\Big|_{z=1}}{(n+1)\pi Q^{n}\1 }\overset{\eqref{P_eqn1},\eqref{P_eqn2}}%
{=}\frac{(0,\dots,\pi_{j},\dots,0)Q^{n}\1 +\pi(n\rho^{n}u_{j}%
v_{j}vu^{\top}+O(\rho^{n}))\1 }{(n+1)\pi(\rho^{n}vu^{\top}%
+O(n^{d-1}|\rho_{1}|^{n}))\1 }\\
&  =\frac{\pi(\rho^{n}u_{j}v_{j}vu^{\top})\1 }{\pi(\rho^{n}vu^{\top
})\1 }+O(\tfrac{1}{n})=u_{j}v_{j}+O(\tfrac{1}{n})\,.
\end{align*}
(iii) If $Q$ is cyclic with period $h$, then by Seneta \cite[Theorem
1.4]{Sene06} the matrix $Q^{h}$ is eventually positive. Hence by assertion
(ii)%
\[
\mathbb{E}_{\pi}\left[  \tfrac{1}{n+1}\#\big\{m\in\{0,1,\dots,n\}:X_{mh}%
=j\big\}\big|T>nh\right]  =u_{j}v_{j}+O(\tfrac{1}{n})\,
\]
for the right and left normalised eigenvectors $v$ and $u^{\top}$ of $Q$ for
the eigenvalue $\rho$, which are also eigenvectors for $Q^{h}$ for the
eigenvalue $\rho^{h}$. Naturally, $O(\tfrac{1}{nh})=O(\tfrac{1}{n})$ for
$n\to\infty$.

One also finds for every $\ell\in\{1,\dots,h-1\}$%
\[
\mathbb{E}_{\pi}\left[  \tfrac{1}{n+1}\#\big\{m\in\{0,1,\dots,n-1\}:X_{mh+\ell
}=j\big\}\big|T>nh+i\right]  =u_{j}v_{j}+O(\tfrac{1}{n})\,
\]
Summing for $\ell=0,\dots,h-1$ one finds%
\begin{align*}
&  \mathbb{E}_{\pi}\left[  \frac{1}{nh}\#\big\{\left\{  m\in\{0,\dots
,nh\}:X_{m}=j\big\}\big|T>nh\right\}  \right] \\
&  =\mathbb{E}_{\pi}\left[  \frac{1}{nh}\sum\limits_{\ell=0}^{h-1}\left[
\#\big\{m\in\{0,1,\dots,(n-1)\}:X_{mh+\ell}=j\big\}\big|T>nh\right]  \right]
\\
&  =\frac{1}{h}\sum\limits_{\ell=0}^{h-1}\mathbb{E}_{\pi}\left[  \frac{1}%
{n}\left[  \#\big\{m\in\{0,1,\dots,(n-1)\}:X_{mh+\ell}%
=j\big\}\big|T>nh\right]  \right] \\
&  =\frac{1}{h}\sum\limits_{\ell=0}^{h-1}\left(  u_{j}v_{j}+O(\tfrac{1}%
{n})\right)  =u_{j}v_{j}+O(\tfrac{1}{n}).
\end{align*}
This concludes the proof of Proposition \ref{P_prop1}.
\end{proof}


\begin{thebibliography}{99}                                                                                               %
\bibitem {BenaCP16}M. Bena\"{\i}m, B. Cloez, and F. Panloup, Stochastic
approximation of quasi-stationary distributions on compact spaces and
applications, Annals of Applied Probability 28(4), (2018), 2370--2416.

\bibitem {BreyR99}L.A. Breyer and G.O. Roberts, A quasi-ergodic theorem for
evanescent processes, Stochastic Process Appl. 84 (1999), 177--186.

\bibitem {ChamV16}N. Champagnat and D. Villemonais, Exponential convergence to
quasi-stationary distribution and Q-process, Probab. Theory Related Fields 164
(2016), 243--283.

\bibitem {ChamV18}N. Champagnat and D. Villemonais, General criteria for the
study of quasi-stationarity, arXiv: 1712.08092v2, 2018.

\bibitem {ColMS13}P. Collet, S. Martinez, and J. San Martin, Quasi-Stationary
Distributions: Markov Chains, Diffusions, and Dynamical Systems,
Springer-Verlag, 2013.

\bibitem {DarrS65}J.N. Darroch and E. Seneta, Distributions in absorbing
discrete-time finite Markov chains, J. Appl. Prob. 2(1) (1965), 88--100.

\bibitem{Engel_19_1} M.~Engel, J.S.W.~Lamb, and M.~Rasmussen, Conditioned Lyapunov exponents for random dynamical systems, Transactions of the American Mathematical Society 372(9) (2019), 6343--6370.

\bibitem {FriSchn80}S. Friedland and H. Schneider, The growth of powers of a
nonnegative matrix, SIAM J. Alg. Disc. Meth. 1(2) (1980), 185--200.

\bibitem {Gant59}F.R. Gantmacher, The Theory of Matrices, Volume II, Chelsea, 1959.

\bibitem {HeZha16}G. He and H. Zhang, On quasi-ergodic distribution for
one-dimensional diffusions, Statistics and Probability Letters 110 (2016), 175--180.

\bibitem {HeZhaZhu19}G. He, H. Zhang, and Y. Zhu, On the quasi-ergodic
distribution of absorbing Markov processes, Statistics and Probability Letters
149 (2019), 116--123.

\bibitem {Meleard12}S. M\'{e}l\'{e}ard and D. Villemonais, Quasi-stationary
distributions and population processes, Probability Surveys 9 (2012), 340--410.

\bibitem {Schn86}H. Schneider, The influence of the marked reduced graph of a
nonnegative matrix on the Jordan form and on related properties:\ a survey,
Linear Algebra and its Applications 84 (1986), 161--189.

\bibitem {Sene06}E. Seneta, Non-negative Matrices and Markov Chains, revised
printing, Springer 2006.

\bibitem {vanDooPol09}E.A.~van Doorn and P.K.~Pollett, Quasi-stationary
distributions for reducible absorbing Markov chains in discrete time, Markov
Processes and Related Fields 15 (2009), 191--204.

\bibitem {ZhanLS14}J. Zhang, S. Li, and R. Song, Quasi-stationarity and
quasi-ergodicity of general Markov processes, Science China Mathematics 57(10)
(2014), 2013--2024.
\end{thebibliography}
\end{document}